\documentclass[11pt,reqno]{amsart}

\usepackage{stix2}
\usepackage[colorlinks]{hyperref}
\usepackage{amssymb}
\usepackage{bbm}
\usepackage[all,cmtip]{xy}
\usepackage{cite}

\newcommand{\D}{{\mathbb D}}
\newcommand{\C}{{\mathbb C}}
\newcommand{\R}{{\mathbb R}}
\newcommand{\T}{{\mathbb T}}
\newcommand{\ind}{\int_{\D}}
\newcommand{\aut}{{\rm Aut}\,(\D)}
\newcommand{\inc}{\int_{\C}}
\newcommand{\inr}{\int_{\R}}

\newtheorem{thm}{Theorem}[section]
\newtheorem{cor}[thm]{Corollary}
\newtheorem{lem}[thm]{Lemma}
\newtheorem{prop}[thm]{Proposition}

\theoremstyle{definition}
\newtheorem{defn}[thm]{Definition}

\numberwithin{equation}{section}

\newtheorem*{cor*}{Corollary}

\allowdisplaybreaks

\title[Localization operators]
{Localization operators on Bergman and Fock spaces}

\author{Pan Ma}
\address[Pan Ma]{School of Mathematics and Statistics, HNP-LAMA,
Central South University, Changsha, Hunan 410083, China}
\email{pan.ma@csu.edu.cn}

\author{Fugang Yan}
\address[Fugang Yan]{College of Mathematics and Statistics,
Chongqing University, Chongqing 401331, China; and Key
Laboratory of Nonlinear Analysis and its Applications (Chongqing
University), Ministry of Education, Chongqing 401331, China}
\email{fugangyan@cqu.edu.cn}

\author{Dechao Zheng}
\address[Dechao Zheng]{Department of Mathematics,
Vanderbilt University, Nashville, TN 37240, USA.}
\email{dechao.zheng@vanderbilt.edu}

\author{Kehe Zhu}
\address[Kehe Zhu]{Department of Mathematics and Statistics, SUNY,
Albany, NY 12222, USA.}
\email{kzhu@albany.edu}


\makeatletter
\@namedef{subjclassname@2020}{\textup{2020} Mathematics Subject Classification}
\makeatother

\begin{document}

\begin{abstract}
We introduce localization operators on weighted Bergman and
Fock spaces and show that, under a natural scaling of symbols and
window functions, localization operators on the weighted
Bergman space $A_{\beta r^2}^2$ converge, in the weak sense,
to localization operators on the Fock space $F_{\beta}^2$
as $r\to\infty$. From this we derive several applications, including
one about sharp norm estimates for certain Toeplitz operators
on Fock spaces, one about windowed Berezin transforms for
weighted Bergman spaces, and another about Szeg\"{o}-type
theorems for localization operators on weighted Bergman spaces.
\end{abstract}

\maketitle

\section{Introduction}

Localization operators, first introduced by Daubechies \cite{Da} in
time-frequency analysis, serve as a mathematical tool for localizing a
signal in phase space. They play an important role in signal processing,
quantum mechanics, and related areas \cite{FN2, Wo1, Wo2}. More
precisely, the time-frequency localization operator $L_{f}^{\phi,\psi}$
on $L^2(\R)=L^2(\R, dx)$, associated with a symbol function $f$ on
$\R^2$ and windows $\phi,\psi\in L^2(\R)$, is defined by the
sesquilinear form
$$\langle L_{f}^{\phi,\psi}g,\,h\rangle=\inr d\omega\inr f(x,\omega)
\langle g,\,M_{\omega}T_x\phi\rangle
\langle M_{\omega}T_x\psi,\,h\rangle\,dx,$$
where $\langle\ ,\ \rangle$ is the inner product in $L^2(\R)$, $T_x$
is the translation operator, and $M_\omega$ is the modulation operator:
$$T_xf(t)=f(t-x),\quad M_{\omega}f(t)=e^{2\pi i\omega t}f(t),
\qquad f\in L^2(\R).$$
Numerous properties of time-frequency localization operators,
including boundedness, compactness, and spectral behavior, have
been extensively studied in the literature; see
\cite{BOW, CG, DNW, FG, RT1, RT2}.

In this paper, we introduce and study localization
operators on Fock spaces of the complex plane and weighted
Bergman spaces (with standard radial weights) of the unit disc.

Recall that, for any positive parameter $\beta$, the Fock space
$F^2_\beta$ of the complex plane $\C$ consists of all entire
functions $f$ with
$$\|f\|^2_{F_{\beta}^2}=\int_{\C}|f(z)|^2d\mu_{\beta}(z)
<\infty,$$
where
$$d\mu_{\beta}(z)=\frac{\beta}{\pi}e^{-\beta|z|^2}dA(z)$$
is the Gaussian measure. Here $dA$ is area measure on
the complex plane $\C$. Note that $F^2_\beta$ is a
reproducing kernel Hilbert space with kernel function
$e^{\beta z\overline w}$.

Also recall that, for any parameter $\alpha>-1$, the weighted
Bergman space $A^2_\alpha$ of the open unit disc
$\D=\{z\in\C: |z|<1\}$ consists of all analytic functions $f$
on $\D$ such that
$$\|f\|^2_{A^2_\alpha}=\ind|f(z)|^2\,dA_\alpha(z)<\infty,$$
where
$$dA_\alpha(z)=\frac{\alpha+1}{\pi}(1-|z|^2)^\alpha\,dA(z).$$
$A^2_\alpha$ is also
a reproducing kernel Hilbert space with kernel function
$(1-z\overline w)^{-(2+\alpha)}$. When $\alpha=0$, we
simply write $A^2$ instead of $A^2_0$. It is well known and
easy to verify that the fractional differential operator
$V^\alpha$ defined by
$$V^\alpha f(z)=\sum_{n=0}^\infty\frac{a_n}{\sqrt{n+1}}
\sqrt{\frac{\Gamma(n+\alpha+2)}{n!\,\Gamma(\alpha+2)}}
\,z^n,\qquad f(z)=\sum_{n=0}^\infty a_n\,z^n,$$
is a unitary transformation from $A^2$ to $A^2_\alpha$.

Roughly speaking, the unitary translation and modulation
operators on $L^2(\R)$ correspond to the so-called Weyl
unitary operators on $F^2_\beta$ and a similar family of
unitary operators on $A^2_\alpha$. More specifically, for
any $\beta>0$ and any $z\in\C$, the Weyl unitary
operator $W^\beta_z$ on $F^2_\beta$ is defined by
$$W^\beta_zf(\zeta)=f(\zeta-z)e^{\beta\overline z\zeta
-(\beta|z|^2/2)}.$$
Similarly, for any $\alpha>-1$ and any $z\in\D$, we can
define a unitary operator $U^\alpha_z$ on $A^2_\alpha$ by
$$U^{\alpha}_zf(\zeta)=f\circ\varphi_z(\zeta)\left[\frac{1-|z|^2}
{(1-\overline z\zeta)^2}\right]^{(2+\alpha)/2},$$
where $\varphi_z(\zeta)=(\zeta-z)/(1-\overline z\zeta)$ is a
M\"obius map of the unit disc with $\varphi_z^{-1}=
\varphi_{-z}$.

The Weyl operators $W^\beta_z$ constitute the main part
of the Weyl unitary representation of the Heisenberg group
on $F^2_\beta$. The remaining part consists roughly of
rotations. Similarly, the operators $U_z^\alpha$ constitute the
main part of a natural unitary representation of the M\"obius
group of $\D$ on the Hilbert space $A^2_\alpha$, and the
remaining part roughly corresponds to rotations again.
We will use $\T=\partial\D$ to denote the unit circle,
which is used to represent rotations for both $\C$ and
$\D$. We now define localization operators on $F^2_\beta$
and $A^2_\alpha$ as follows.

\begin{defn}\label{Fock localization}
Let $\phi,\psi\in F^2_\beta$ and $f\in L^\infty(\T\times\C)$.
We define a linear operator $\mathbb L_f^{\phi,\psi,\beta}$
on $F^2_\beta$ by the sesquilinear form
$$\langle \mathbb L_f^{\phi,\psi,\beta}g,h\rangle=
\frac\beta\pi\int_0^{2\pi}\frac{d\theta}{2\pi}
\inc f(e^{i\theta},z)\langle g, W^\beta_z\phi_\theta
\rangle\langle W^\beta_z\psi_\theta,h\rangle\,dA(z),$$
where $\langle\ ,\ \rangle$ is the inner product in $F^2_\beta$
and $\phi_\theta(\zeta)=\phi(e^{i\theta}\zeta)$. We will call
$\mathbb L_f^{\phi,\psi,\beta}$ a localization operator on
$F^2_\beta$, with symbol $f$ and windows $\phi$ and $\psi$.
When $\phi=\psi$, we simply write
$\mathbb L_f^{\psi,\,\beta}=\mathbb L_f^{\psi,\psi,\,\beta}$.
\end{defn}

\begin{defn}\label{Bergman localization}
Let $\phi,\psi\in A^2_\alpha$ and $f\in L^\infty(\T\times\D)$.
We define a linear operator $\mathbf L_f^{\phi,\psi,\alpha}$
on $A^2_\alpha$ by the sesquilinear form
$$\langle \mathbf L_f^{\phi,\psi,\alpha}g,h\rangle=
(\alpha+1)\int_0^{2\pi}\frac{d\theta}{2\pi}\ind f(e^{i\theta},z)
\langle g, U^\alpha_z\phi_\theta\rangle\langle U^\alpha_z
\psi_\theta,h\rangle\,d\lambda(z),$$
where $\langle\ ,\ \rangle$ is the inner product in
$A^2_\alpha$, $\phi_\theta(\zeta)=\phi(e^{i\theta}\zeta)$, and
$$d\lambda(z)=\frac{dA(z)}{\pi(1-|z|^2)^2}$$
is the so-called M\"obius invariant area measure on $\D$. Again,
we will call $\mathbf L_f^{\phi,\psi,\alpha}$ a localization operator
on $A^2_\alpha$, with symbol $f$ and windows $\phi$ and $\psi$.
When $\phi=\psi$, we simply write
$\mathbf L_f^{\psi,\,\alpha}=\mathbf L_f^{\psi,\psi,\,\alpha}$.
\end{defn}

Our first main result is the following.

\begin{thm}\label{thma}
Suppose $\phi,\psi\in F_{\beta}^2$ with $\beta>0$ and
$f\in L^{\infty}(\T\times\C)$. For any $\sigma\geq 0$ we have
$$\lim_{r\to\infty}\left\langle
\mathbf L_{f_{r,\,\sigma}}^{\phi_r,\psi_r,\,\beta r^2}g_r,\,h_r
\right\rangle_{A^2_{\beta r^2}}=\left\langle
\mathbb L_f^{\phi,\psi,\beta}g,\,h\right\rangle_{F_{\beta}^2},
\qquad g,h\in F^2_{\beta},$$
where $\phi_r(z)=\phi(rz)$ and
$$f_{r,\,\sigma}(e^{i\theta},z)=(1-|z|^2)^{\sigma}
f(e^{i\theta},rz),\qquad z\in\D.$$
\end{thm}

Recall that, for any $f\in L^\infty(\C)$, the Toeplitz operator
$\mathbb T^\beta_f: F^2_\beta\to F^2_\beta$ is defined by
$$\mathbb T^\beta_fg(z)=\inc f(w)g(w)e^{\beta z\overline w}
\,d\mu_\beta(w).$$
Similarly, for any $f\in L^\infty(\D)$, the Toeplitz operator
$\mathbf T^\alpha_f: A^2_\alpha\to A^2_\alpha$ is defined by
$$\mathbf T^\alpha_fg(z)=\ind\frac{f(w)g(w)}
{(1-z\overline w)^{2+\alpha}}\,dA_\alpha(w).$$
The following sharp inequality for the norm of Toeplitz
operators on weighted Bergman spaces appeared
implicitly in \cite{NT}: if $f\in L^1(\D,d\lambda)\cap
L^{\infty}(\D)$, then
\begin{equation}\label{eq1.1}
\|\mathbf T_{f}^{\alpha}\|_{A_{\alpha}^2}\leq\left(1-
\frac{\|f\|_{\infty}^{\alpha+1}}{(\|f\|_{\infty}+
\|f\|_{L^1(\D,\,d\lambda)})^{\alpha+1}}\right)\|f\|_{\infty}\,,
\end{equation}
and equality holds if $f$ is the characteristic function
of a hyperbolic disc in $\D$.

If $\phi=\psi=1$ and $f(e^{i\theta},z)=f(z)$ is independent
of $\theta$, it is clear that
$$\mathbb L_f^{\phi,\psi,\beta}=\mathbb T^\beta_f\qquad
{\rm and}\qquad \mathbf L^{\phi,\psi,\alpha}_f
=\mathbf T^\alpha_f.$$
This together with Theorem~\ref{thma} and \eqref{eq1.1}
yields the following new norm estimate for Toeplitz operators
on the Fock space.

\begin{cor}
For $f\in L^1(\C)\cap L^{\infty}(\C)$ and $\beta>0$, we have
$$\|\mathbb T^{\beta}_{f}\|_{F_{\beta}^2}\leq
\left[1-\exp\left(-\frac{\beta}{\pi}\frac{\|f\|_{L^1(\C)}}
{\|f\|_{\infty}}\right)\right]\,\|f\|_{\infty}\,,$$
and equality holds if $f$ is the characteristic function of an
Euclidean disc in $\C$. Here $L^1(\C)=L^1(\C, dA)$.
\end{cor}

Some special cases of this corollary with $\beta=\pi$ were known
in the literature. For example, Galbis \cite{Ga} obtained the result
for real-valued and radial symbols, while Huang-Zhang \cite{HZ}
treated the case of real-valued symbols.

Our next result concerns the limit behavior of windowed
Berezin transforms on $A^2_\alpha$. Recall that the ordinary
(unwindowed) Berezin transform associated with $A^2_\alpha$
is the integral operator
$$B_\alpha f(z)=\ind f(w)\frac{(1-|z|^2)^{2+\alpha}}
{|1-\overline zw|^{2(2+\alpha)}}\,dA_\alpha(w)=
\ind f(w)\left|U^\alpha_z1(w)\right|^2\,dA_\alpha(w).$$
The Berezin transform is an important and useful tool in
operator theory of holomorphic function spaces;
see \cite{AFR, AZ, En, MSW, Su, Zhu}.

Given a window function $\psi\in A^2_\alpha$, we
define the windowed Berezin transform $B^\psi_\alpha$
as the following integral operator:
$$B^\psi_\alpha f(e^{i\theta},z)=(\alpha+1)
\int_0^{2\pi}\frac{dt}{2\pi}\ind f(e^{it},w)
\left|\langle U^\alpha_z\psi_\theta,U^\alpha_w\psi_t
\rangle_{A_{\alpha}^2}\right|^2\,d\lambda(w),$$
where $\psi_t(z)=\psi(e^{it}z)$ like before. It is clear that if
$\psi=1$ and $f(e^{i\theta},z)=f(z)$ is independent of $\theta$,
then the windowed Berezin transform reduces to the ordinary
Berezin transform.

We can now state the second main result of the paper.

\begin{thm}\label{thmb}
Suppose $\psi\in A^2$ with $\|\psi\|_{A^2}=1$,
$\psi^\alpha=V^\alpha\psi$, and $1\leq p<\infty$. For
$f\in L^p(\D,d\lambda)$, we have
$$\lim_{\alpha\to\infty}\|B_{\alpha}^{\psi^{\alpha}}f-
f\|_{L^p(\T\times \D)}=0,$$
where $L^p(\T\times\D)=L^p(\T\times \D,\frac{d\theta}{2\pi}\,d\lambda)$.
\end{thm}

As an application of Theorem~\ref{thmb}, we obtain the following
Szeg\"{o}-type theorem, the third main result of the paper, for
localization operators on weighted Bergman spaces.

\begin{thm}\label{thmc}
Suppose $\psi$ is a unit vector in $A^2$, $\psi^\alpha=V^\alpha
\psi$, and $f$ is a non-negative function in $L^1(\D,d\lambda)
\cap L^{\infty}(\D)$. If $h$ is continuous on the closed interval
$[0,\|f\|_{\infty}]$, then
$$\lim_{\alpha\to\infty}\frac{{\rm tr}(\mathbf L^{\psi^{\alpha},
\alpha}_fh(\mathbf L^{\psi^{\alpha},\alpha}_f))}{\alpha+1}
=\ind f(z)h(f(z))\,d\lambda(z).$$
\end{thm}

The special case $\psi=1$ has been obtained by Camper and Mitkovski
in \cite{CM}. We also refer to \cite{FN1} for a Szeg\"{o}-type theorem
of Gabor-Toeplitz localization operators. Two interesting consequences of
Theorem~\ref{thmc} are given below.

\begin{cor}
Suppose $\psi$ is a unit vector in $A^2$, $\psi^\alpha=V^\alpha\psi$,
and $f$ is a non-negative function in $L^1(\D,d\lambda)\cap
L^{\infty}(\D)$. For $0<\delta\leq\|f\|_{\infty}$ we have
$$\lim_{\alpha\to\infty}\frac{\#\left\{i:\lambda_i(
\mathbf L_{f}^{\psi^{\alpha},\alpha})>\delta\right\}}{\alpha+1}
=\lambda(\{z\in\D:f(z)>\delta\}),$$
where $\lambda_i(T)$ denotes the $i$-th singular value of $T$.
\end{cor}

\begin{cor}
Suppose $\psi$ is a unit vector in $A^2$, $\psi^\alpha=
V^\alpha\psi$, and $f$ is a non-negative function in
$L^1(\D,d\lambda)\cap L^{\infty}(\D)$. Then we have
$$\lim_{\alpha\to \infty}\|\mathbf L_f^{\psi^{\alpha},\alpha}
\|_{A_{\alpha}^2}=\|f\|_{\infty}.$$
\end{cor}

\section{The Fock space as a limit of weighted Bergman spaces}

The Fock space $F^2_\beta$ is an analytic function space on the complex
plane. In this section we will show that $F^2_\beta$ is the weak limit of
certain weighted Bergman spaces of the unit disc. Although this
has been known to experts in the area (for example, it was mentioned in
\cite{JPR} and it was given as an exercise in \cite{Zhu1}), details are
difficult to find in the literature. So we will give a full proof for
the following result.

\begin{thm}\label{2.1}
If $f\in L^1(\C,d\mu_{\beta})$, then for any $\sigma\geq 0$ we have
$$\lim_{r\to\infty}\ind f(rz)\,dA_{\beta r^2+\sigma}(z)
=\inc f(z)\,d\mu_{\beta}(z).$$
\end{thm}

\begin{proof}
By a change of variables, we have
\begin{align}
\ind f(rz)\,dA_{\beta r^2+\sigma}(z)
&=(\beta r^2+\sigma+1)\ind f(rz)(1-|z|^2)^{\beta r^2+\sigma}
\frac{dA(z)}{\pi}\nonumber\\
&=\frac{(\beta r^2+\sigma+1)}{r^{2}}\int_{r\D}f(z)\left(1-\left|\frac{z}{r}
\right|^2\right)^{\beta r^2+\sigma}\frac{dA(z)}{\pi}\nonumber\\
&=\frac{(\beta r^2+\sigma+1)}{r^{2}\pi}\inc f(z)\psi_r(z)\,dA(z)\label{eq2.1}
\end{align}
for $r>0$, where
$$\psi_r(z)=\mathbbm{1}_{r\D}(z)\bigg(1-\frac{|z|^2}{r^2}
\bigg)^{\beta r^2+\sigma}$$
and $\mathbbm{1}_{r\D}(z)$ is the characteristic function of $r\D$. Clearly,
\begin{equation}\label{eq2.2}
\lim_{r\to\infty}\frac{(\beta r^2+\sigma+1)}{r^{2}\pi}=\frac{\beta}{\pi}.
\end{equation}
It is easy to see that $\log(1-x)\leq -x$ for $x\in(0,1)$. This implies
\begin{equation}\label{eq2.3}
(1-x)^{\frac{1}{x}}\leq e^{-1},\quad x\in (0,1).
\end{equation}
It follows from \eqref{eq2.3} that for $z\in\C$ we have
\begin{eqnarray*}
\psi_r(z)&=&\mathbbm {1}_{r\D}(z)\left(1-\left|\frac{z}{r}\right|^2
\right)^{\beta r^2+\sigma}\leq\mathbbm {1}_{r\D}(z)\left(1-\left|
\frac{z}{r}\right|^2\right)^{\beta r^2}\\
&=&\mathbbm {1}_{r\D}(z)\left[\left(1-\left|\frac{z}{r}\right|^2
\right)^{\left|\frac{r}{z}\right|^2}\right]^{\beta|z|^2}\leq
\mathbbm {1}_{r\D}(z) e^{-\beta|z|^2}\leq e^{-\beta|z|^2}.
\end{eqnarray*}
Since $f$ is in $L^1(\C,d\mu_{\beta})$ and $\psi_r(z)$ converges to
$e^{-\beta|z|^2}$ as $r$ goes to $\infty$, we deduce from
\eqref{eq2.1}, \eqref{eq2.2}, and the dominated convergence
theorem that
$$\lim_{r\to \infty}\ind f(rz)dA_{\beta r^2+\sigma}(z)=
\frac{\beta}{\pi}\inc f(z)e^{-\beta|z|^2}\,dA(z).$$
This completes the proof of the theorem.
\end{proof}

\begin{cor}\label{2.2}
Let $\beta>0$ and $\sigma\geq 0$. For any entire function
$f$ on $\C$ we have
$$\lim_{r\to\infty}\|f_r\|_{A_{\beta r^2+\sigma}^2}
=\|f\|_{F^2_{\beta}},$$
where $f_r(z)=f(rz)$.
\end{cor}

\section{Orthogonality relations in $A^2_{\alpha}$ and $F_{\beta}^2$}

Recall that the orthogonality relation in time-frequency analysis
refers to the following identity,
$$\int_{\R^2}\langle g,\,M_{\omega}T_x\phi\rangle
\langle M_{\omega}T_x\psi,\,h\rangle\,dxd\omega=
\langle g,\,h\rangle\langle \psi,\,\phi \rangle,$$
where $g,h,\phi,\psi\in L^2(\R)$ and the inner product is that
of $L^2(\R)$. This is often referred to as Moyal's identity;
see \cite[Theorem 3.2.1]{Gr}.The associated localization operators
on $L^2(\R)$ are operators $L^{\phi,\psi}_f$ on $L^2(\R)$ defined by
$$\langle L^{\phi,\psi}_fg,h\rangle=\int_{\R^2}f(x,\omega)
\langle g, M_\omega T_x\phi\rangle\langle M_\omega T_x\psi,
h\rangle\,dx\,d\omega,$$
where $\phi$ and $\psi$ are window functions in $L^2(\R)$, $f$ is a
symbol function on $\R^2$, and the inner product is taken in $L^2(\R)$.

In this section we will extend the orthogonality relation and the
associated localization operators above to the settings of weighted
Bergman spaces and Fock spaces.

For the Fock space $F^2_\beta$, $\beta>0$, we begin with the classical
Bargmann transform, denoted by ${\mathcal B}_{\beta}$, which is a
certain integral operator mapping $L^2(\R)$ unitarily onto $F_{\beta}^2$.
The Bargmann transform allows us to transfer functions in $L^2(\R)$
to functions in $F^2_\beta$, and to transfer operators on $L^2(\R)$
to operators on $F^2_\beta$. For example, it is well known that
\begin{equation}\label{eq3.1}
{\mathcal B}_{\beta}M_{\omega}T_x{\mathcal B}_{\beta}^{-1}
=e^{\beta ix\omega}W^{\beta}_{x-\frac{\pi \omega i}{\beta}}.
\end{equation}
See \cite{Zhu1} for more information about the Fock space,
Weyl unitary operators, and the Bargmann transform.

By \eqref{eq3.1}, the Bargmann transform takes the orthogonality
relation in $L^2(\R)$ to $F^2_\beta$ as follows:{\small
\begin{align*}
\int_{\R^2}\langle {\mathcal B}_\beta g, W_{x-i(\pi\omega/\beta)}
{\mathcal B}_\beta\phi\rangle\langle
W_{x-i(\pi\omega/\beta)}{\mathcal B}_\beta\psi,
{\mathcal B}_\beta h\rangle\,dx\,d\omega
=\langle {\mathcal B}_\beta g, {\mathcal B}_\beta h
\rangle\langle {\mathcal B}_\beta\phi,
{\mathcal B}_\beta\psi\rangle,
\end{align*}
}where $\langle\ ,\ \rangle$ is the inner product in $F^2_\beta$.
Replacing ${\mathcal B}_\beta g,{\mathcal B}_\beta h,
{\mathcal B}_\beta\phi,{\mathcal B}_\beta\psi$ by $g$,$h$,
$\phi$,$\psi\in F^2_\beta$, and writing $z=x-i(\pi\omega/\beta)$,
we obtain the following Fock space analog of Moyal's identity.

\begin{thm}\label{3.1}
If $\phi,\psi,g,h\in F_{\beta}^2$, then we have
$$\frac{\beta}{\pi}\inc\langle g,\,W^{\beta}_z\phi
\rangle\langle W^{\beta}_z\psi,\,h\rangle\,dA(z)=
\langle g,\,h\rangle\langle\psi,\,\phi\rangle.$$
\end{thm}

Note that a direct proof of Theorem~\ref{3.1} without using the
Bargmann transform is possible. In fact, for $z\in\C$, if we write
$K_{z}^{\beta}(\zeta)=e^{\beta\overline{z}\zeta}$
for the reproducing kernel of $F_{\beta}^2$ at $z$, then
an easy computation shows that
$$\langle K^{\beta}_{\zeta},\,W^{\beta}_zK^{\beta}_{z_1}
\rangle=\overline{(W^{\beta}_zK^{\beta}_{z_1})(\zeta)}
=e^{-\frac{\beta}{2}|z|^2}e^{-\beta z_1\bar z}
e^{\beta(z_1+z)\bar{\zeta}},$$
and
$$\langle W^{\beta}_zK^{\beta}_{z_2},\,K^{\beta}_{\xi}
\rangle=(W^{\beta}_zK^{\beta}_{z_2})(\xi)=
e^{-\frac{\beta}{2}|z|^2}e^{-\beta \bar z_2 z}
e^{\beta(\bar z_2+\bar z){\xi}}.$$
It follows that
\begin{align*}
\frac{\beta}{\pi}&\int_{\C}\langle K^{\beta}_{\zeta},
\,W^{\beta}_zK^{\beta}_{z_1}\rangle
\langle W^{\beta}_zK^{\beta}_{z_2},\,K^{\beta}_{\xi}
\rangle\,dA(z)\\
&=e^{\beta z_1\bar{\zeta}+\beta \bar z_2\xi}
\inc e^{\beta(\bar \zeta-\bar z_2)z}e^{\beta(\xi-z_1)\bar z}
\,d\mu_{\beta}(z)\\
&=e^{\beta\bar \zeta\xi+\beta\bar z_2z_1}=
\langle K^{\beta}_{\zeta},\,K^{\beta}_{\xi}
\rangle\langle K^{\beta}_{z_2},\,K^{\beta}_{z_1}\rangle.
\end{align*}
As a consequence, the equality in Theorem \ref{3.1} holds for
all $\phi,\psi,g,h\in \mathcal D_{F_{\beta}^2}$, where
$$\mathcal D_{F_{\beta}^2}=
{\rm span}\{K^{\beta}_z:z\in\C\},$$
which is dense in $F^2_\beta$.

Similar to the orthogonality relation, we can also use the
Bargmann transform to show that the localization operator
$L^{\phi,\psi}_f$ on $L^2(\R)$ is unitarily equivalent to the
operator $L$ on $F^2_\beta$ defined via the sesquilinear form
\begin{equation}\label{eq?}
\langle Lg,\,h\rangle=\frac{\beta}{\pi}\inc f(z)\langle g,
\,W^{\beta}_z\phi\rangle\langle W^{\beta}_z\psi,\,h
\rangle\,dA(z),\quad g,h\in F_{\beta}^2.
\end{equation}
Here we abuse notion, use $g,h,\phi,\psi\in F^2_\beta$ for
the Bargmann transforms of the original functions
$g,h,\phi,\psi\in L^2(\R)$, and use $f(z)$ for the
orginal $f(x,-\beta\omega/\pi)$.

This operator $L$ is somewhat inconsistent with the localization
operators on $L^2(\R)$. More specifically, the time-frequency
localization operator $L^{\phi,\psi}_f$ is defined for a symbol
function $f$ on $\R\times\R$, instead of $\R$ which is the
underlying measure space for the Hilbert space $L^2(\R)$. But
the operator $L$ above is based on a symbol function $f$ on $\C$,
which is the underlying measure space for the Hilbert space
$F^2_\beta$. To rectify this discrepancy and in order to achieve
consistency with the Bergman space setting, we note that the
rotation-invariance of the Gaussian measure implies that the
orthogonality relation in Theorem~\ref{3.1} can be rewritten as
$$\frac{\beta}{\pi}\inc\langle g,\,W^{\beta}_z\phi_\theta
\rangle\langle W^{\beta}_z\psi_\theta,\,h\rangle\,dA(z)=
\langle g,\,h\rangle\langle\psi_\theta,\,\phi_\theta\rangle
=\langle g,\,h\rangle\langle\psi,\,\phi\rangle,$$
where $\phi_\theta(\zeta)=\phi(e^{i\theta}\zeta)$ with $\theta
\in[0,2\pi]$ and the inner product is that of $F^2_\beta$.
Consequently,
$$\frac{\beta}{\pi}\int_0^{2\pi}\frac{d\theta}{2\pi}
\inc\langle g,\,W^{\beta}_z\phi_\theta
\rangle\langle W^{\beta}_z\psi_\theta,\,h\rangle\,dA(z)
=\langle g,\,h\rangle\langle\psi,\,\phi\rangle.$$
This motivates our Definition~\ref{Fock localization} for
localization operators ${\mathbb L}^{\phi,\psi}_f$ on
$F^2_\beta$, namely,
$$\langle \mathbb L_f^{\phi,\psi,\beta}g,h\rangle=
\frac\beta\pi\int_0^{2\pi}\frac{d\theta}{2\pi}
\inc f(e^{i\theta},z)\langle g, W^\beta_z\phi_\theta
\rangle\langle W^\beta_z\psi_\theta,h\rangle\,dA(z),$$
where the inner product is that of $F^2_\alpha$ and the
symbol function $f$ is defined on $\T\times\C$, instead
of $\C$. The operator ${\mathbb L}^{\phi,\psi}_f$ above
appears more aligned with the time-frequency localization
operator $L^{\phi,\psi}_f$ than the operator $L$ in \eqref{eq?}.

Next we turn our attention to weighted Bergman spaces. Let
$\aut$ denote the M\"obius group of the unit disc, consisting
of all bijective analytic maps of $\D$. Recall that every
$\varphi\in\aut$ can be written in the form
$\varphi(z)=e^{i\theta}\varphi_a(z)$, where
$e^{i\theta}\in\T$, $a\in\D$, and
$$\varphi_a(z)=\frac{z-a}{1-\overline az},\qquad z\in\D.$$
It is easy to see that $e^{i\theta}$ and $a$ are uniquely
determined by $\varphi$. So we can identify $\aut$
topologically with $\T\times\D$. We will first clarify what
operation on $\T\times\D$ corresponds to composition
(the group operation) on $\aut$. For convenience, we write
$\varphi_{e^{i\theta},a}(z)=e^{i\theta}\varphi_a(z)$ for
$(e^{i\theta},z)\in \T\times \D$.

\begin{lem}\label{3.2}
For $e^{i\theta},e^{i\eta}\in\T$ and $a,b\in\D$, we have
$$\varphi_{e^{i\theta},a}\circ\varphi_{e^{i\eta},b}(\zeta)
=e^{i(\theta+\eta)}\frac{1+e^{-i\eta}a\bar b}
{1+e^{i\eta}\bar ab}\varphi_{\varphi_{-b}(e^{-i\eta}a)}(\zeta),
\qquad \zeta\in\D. $$
\end{lem}

\begin{proof}
This follows from a direct calculation. In fact,
\begin{align*}
\varphi_{e^{i\theta},a}\circ\varphi_{e^{i\eta},b}(\zeta)
&=e^{i\theta}\frac{\varphi_{e^{i\eta},b}(\zeta)-a}
{1-\bar a\varphi_{e^{i\eta},b}(\zeta)}=e^{i\theta}
\frac{e^{i\eta}(\zeta-b)-a(1-\bar b\zeta)}
{(1-\bar b\zeta)-\bar a e^{i\eta}(\zeta-b)}\\
&=e^{i(\theta+\eta)}\frac{(1+e^{-i\eta}a\bar b)
\zeta-(e^{-i\eta}a+b)}{(1+e^{i\eta}\bar a b)-
(e^{i\eta}\bar a+\bar b)\zeta}\\
&=e^{i(\theta+\eta)}\frac{1+e^{-i\eta}a\bar b}
{1+e^{i\eta}\bar a b}\frac{\zeta-\frac{e^{-i\eta}a+b}
{1+e^{-i\eta}a\bar b}}{1-\frac{e^{i\eta}\bar a+\bar b}
{1+e^{i\eta}\bar a b}\zeta}\\
&=e^{i(\theta+\eta)}\frac{1+e^{-i\eta}a\bar b}
{1+e^{i\eta}\bar ab}\varphi_{\varphi_{-b}
(e^{-i\eta}a)}(\zeta).\qquad\qquad\qquad\qquad\qedhere
\end{align*}
\end{proof}

As a consequence of Lemma \ref{3.2}, we can write
$\aut=\T\times\D$ with the group operation on
$\T\times\D$ given by
\begin{equation}\label{eq3.2}
(e^{i\theta},a)\cdot(e^{i\eta},b)=\left(e^{i(\theta+\eta)}
\frac{1+e^{-i\eta}a\bar b}{1+e^{i\eta}\bar ab},
\varphi_{-b}(e^{-i\eta}a)\right).
\end{equation}
The group unit is $(1,0)$ and the inverse of $(e^{i\theta},a)$
is $(e^{-i\theta}, -e^{i\theta}a)$.

For $\alpha>-1$ we let $K_z^{\alpha}$ denote the reproducing
kernel of $A_{\alpha}^2$, that is,
$$K_z^{\alpha}(w)=\frac{1}{(1-\bar zw)^{2+\alpha}}.$$
The normalized reproducing kernel at $z$ is denoted by
$k_z^{\alpha}$, namely,
$$k_z^{\alpha}(w)=\frac{(1-|z|^2)^{1+\frac{\alpha}{2}}}
{(1-\bar zw)^{2+\alpha}}.$$
For $\varphi\in\aut$ with $\varphi(z)=e^{i\theta}\varphi_a(z)$,
we will use $U^\alpha_\varphi=U^\alpha_{e^{i\theta},a}$ to
denote the unitary operator on $A^2_\alpha$ defined by
$$U^\alpha_{e^{i\theta},a}f(z)=f\circ\varphi(z)
\left[\varphi'(z)\right]^{(2+\alpha)/2}.$$
A simple calculation shows
\begin{equation*}
(U_{e^{i\theta},a}^{\alpha}1)(z)=
[\varphi_{e^{i\theta},a}'(z)]^{1+\frac{\alpha}{2}}
=e^{i(1+\frac{\alpha}{2})\theta}k_a^{\alpha}(z).
\end{equation*}
It follows that
\begin{equation}\label{eq3.3}
U^\alpha_{e^{i\theta},a}f(z)=e^{i(1+\frac\alpha2)\theta}
f(e^{i\theta}\varphi_a(z))k^\alpha_a(z)=
e^{i(1+\frac\alpha2)\theta}U^\alpha_af_\theta(z),
\end{equation}
where $f_\theta(z)=f(e^{i\theta}z)$, and
\begin{equation}\label{eq3.4}
U^{\alpha}_{e^{i\eta},b}U^{\alpha}_{e^{i\theta},a}
=U^{\alpha}_{(e^{i\theta},a)\cdot(e^{i\eta},b)}
=U^{\alpha}_{\left(e^{i(\theta+\eta)}
\frac{1+e^{-i\eta}a\bar b}{1+e^{i\eta}\bar ab},
\varphi_{-b}(e^{-i\eta}a)\right)}.
\end{equation}
Moreover, the Haar measure
of $\aut=\T\times\D$ is given by
$$dH(e^{i\theta},z)=\frac{d\theta}{2\pi}\,d\lambda(z),$$
where
$$d\lambda(z)=\frac{dA(z)}{\pi(1-|z|^2)^2}$$
is the M\"obius invariant area measure on $\D$.

In the remainder of this section, the inner product
$\langle\ ,\ \rangle$ always stands for that of the weighted
Bergman space $A^2_\alpha$.

\begin{lem}\label{3.3}
For $z_1,z_2,\zeta,\xi\in \D$ and $\alpha>-1$, we have
\begin{align}\label{eq3.5}
(\alpha+1)&\int_{0}^{2\pi}\,\frac{d\theta}{2\pi}\ind
\langle K_{\zeta}^{\alpha},\,U_{e^{i\theta},z}^{\alpha}
K^{\alpha}_{z_1}\rangle\langle U^{\alpha}_{e^{i\theta},z}
K^{\alpha}_{z_2},\,K^{\alpha}_{\xi}\rangle
\,d\lambda(z)\nonumber\\[8pt]
&=\langle K^{\alpha}_{\zeta},\,K^{\alpha}_{\xi}\rangle
\langle K^{\alpha}_{z_2},\,K^{\alpha}_{z_1}\rangle.
\end{align}
\end{lem}

\begin{proof}
For fixed $\zeta$ and $\xi$, we observe that both sides of
\eqref{eq3.5} are analytic in $z_1$ and conjugate analytic
in $z_2$. It follows from a well-known fact in the function
theory of several complex variables that we only need to
show that \eqref{eq3.5} is true for $z_1=z_2$. By
multiplying  $(1-|w|^2)^{2+\alpha}$ to both sides
of \eqref{eq3.5} with $w=z_1=z_2$, it is enough to show
\begin{equation}\label{eq3.6}
(\alpha+1)\int_{0}^{2\pi}\frac{d\theta}{2\pi}\ind
\overline{(U^{\alpha}_{e^{i\theta},z}k^{\alpha}_{w})
(\zeta)}(U^{\alpha}_{e^{i\theta},z}k^{\alpha}_{w})
(\xi) d\lambda(z)= K^{\alpha}_{\zeta}(\xi).
\end{equation}
From \eqref{eq3.3} and \eqref{eq3.4} we derive that
\begin{align}
U_{e^{i\theta},z}^{\alpha}k^{\alpha}_w
&=U^{\alpha}_{e^{i\theta},z}U^{\alpha}_{1,w}1=
U^{\alpha}_{e^{i\theta}\frac{1+we^{-i\theta}\bar z}
{1+\bar w e^{i\theta}z},\varphi_{-z}
(e^{-i\theta}w)}1\nonumber\\
&=\left[e^{i\theta}\frac{1+we^{-i\theta}\bar z}
{1+\bar w e^{i\theta}z}\right]^{1+\frac{\alpha}{2}}
k^{\alpha}_{\varphi_{-z}(e^{-i\theta}w)}.\label{eq3.7}
\end{align}
This implies that
\begin{equation}\label{eq3.8}
\left|\overline{(U^{\alpha}_{e^{i\theta},z}k^{\alpha}_{w})
(\zeta)}(U^{\alpha}_{e^{i\theta},z}k^{\alpha}_{w})(\xi)
\right|\leq\frac{(1+|w|)^{2+\alpha}}{(1-|w|)^{2+\alpha}}
\frac{(1-|z|^2)^{2+\alpha}}
{(1-|\zeta|)^{2+\alpha}(1-|\xi|)^{2+\alpha}}.
\end{equation}
Clearly, we have
\begin{equation}\label{eq3.9}
\varphi_{-z}(e^{-i\theta}w)=-e^{-i\theta}
\frac{1+e^{i\theta}z\bar w}{1+\overline{e^{i\theta}z} w}
\varphi_w(-e^{i\theta}z).
\end{equation}
Combining \eqref{eq3.7}, \eqref{eq3.8}, and \eqref{eq3.9},
and applying Fubini's theorem (twice) and a change of
variables (twice), we obtain
\begin{align*}
&\quad(\alpha+1)\int_{0}^{2\pi}\frac{d\theta}{2\pi}\ind
\overline{(U^{\alpha}_{e^{i\theta},z}k^{\alpha}_{w})(\zeta)}
(U^{\alpha}_{e^{i\theta},z}k^{\alpha}_{w})(\xi)\,d\lambda(z)\\
&=(\alpha+1)\int_{0}^{2\pi}\frac{d\theta}{2\pi}\ind
\overline{k^{\alpha}_{\varphi_{-z}(e^{-i\theta}w)}(\zeta)}
k^{\alpha}_{\varphi_{-z}(e^{-i\theta}w)}(\xi)\,d\lambda(z)\\
&=(\alpha+1)\int_{0}^{2\pi}\frac{d\theta}{2\pi}\ind
\overline{k^{\alpha}_{-e^{-i\theta}\frac{1+e^{i\theta}z\bar w}
{1+\overline{e^{i\theta}z} w}\varphi_w(-e^{i\theta}z)}(\zeta)}
k^{\alpha}_{-e^{-i\theta}\frac{1+e^{i\theta}z\bar w}
{1+\overline{e^{i\theta}z} w}\varphi_w(-e^{i\theta}z)}(\xi)\,
d\lambda(z)\\
&=(\alpha+1)\int_{0}^{2\pi}\frac{d\theta}{2\pi}
\ind\overline{k^{\alpha}_{-e^{-i\theta}
\frac{1+z\bar w}{1+\overline{z} w}\varphi_w(-z)}(\zeta)}
k^{\alpha}_{-e^{-i\theta}\frac{1+z\bar w}{1+\overline{z} w}
\varphi_w(-z)}(\xi)\,d\lambda(z)\\
&=(\alpha+1)\ind d\lambda(z)\int_{0}^{2\pi}
\overline{k^{\alpha}_{-e^{-i\theta}
\frac{1+z\bar w}{1+\overline{z} w}\varphi_w(-z)}(\zeta)}
k^{\alpha}_{-e^{-i\theta}\frac{1+z\bar w}{1+\overline{z} w}
\varphi_w(-z)}(\xi)\,\frac{d\theta}{2\pi}\\
&=(\alpha+1)\ind d\lambda(z)\int_{0}^{2\pi}
\overline{k^{\alpha}_{-e^{-i\theta}\varphi_w(-z)}(\zeta)}
k^{\alpha}_{-e^{-i\theta}\varphi_w(-z)}(\xi)
\,\frac{d\theta}{2\pi}\\
&=(\alpha+1)\int_{0}^{2\pi}\frac{d\theta}{2\pi}\ind
\overline{k^{\alpha}_{-e^{-i\theta}\varphi_w(-z)}(\zeta)}
k^{\alpha}_{-e^{-i\theta}\varphi_w(-z)}(\xi)\,d\lambda(z)\\
&=(\alpha+1)\int_{0}^{2\pi}\frac{d\theta}{2\pi}\ind
\overline{k^{\alpha}_{z}(\zeta)}k^{\alpha}_{z}(\xi)\,d\lambda(z)\\
&=\ind K^{\alpha}_{\zeta}(z) \overline{K^{\alpha}_{\xi}(z)}
\,dA_{\alpha}(z)\\
&=K_{\zeta}^{\alpha}(\xi).
\end{align*}
This proves the identity \eqref{eq3.6} and completes the
proof of the lemma.
\end{proof}

The following theorem can be derived from the Schur
orthogonality relations (see \cite[Theorem 14.3.3]{Di} for
example) if we consider the irreducible unitary representation
of $\aut$ defined by $\rho_{\alpha}(e^{i\theta},z)=
U_{e^{i\theta},z}^{\alpha}$. We provide a direct proof here
based on our calculations above for M\"obius maps and
Bergman kernels.

\begin{thm}\label{3.4}
For any $\phi,\psi,g,h\in A_{\alpha}^2$ with $\alpha>-1$ we have
$$(\alpha+1)\int_{0}^{2\pi}\frac{d\theta}{2\pi}\ind\langle g,\,
U^{\alpha}_{z}\phi_\theta\rangle\langle U^{\alpha}_{z}
\psi_\theta,\,h \rangle\,d\lambda(z)=\langle g,\,
h\rangle\langle \psi,\,\phi\rangle,$$
where $\phi_\theta(z)=\phi(e^{i\theta}z)$ and
$$U_z^{\alpha}f(\zeta)=f\circ \varphi_z(\zeta)\left[\varphi_z'(\zeta)\right]^{(2+\alpha)/2}.$$
\end{thm}

\begin{proof}
Let $\mathcal D_{\alpha}$ denote the vector space of all finite
linear combinations of reproducing kernels in $A_{\alpha}^2$.
It is clear that $\mathcal D_{\alpha}$ is dense in $A_{\alpha}^2$.
By Lemma \ref{3.3} and \eqref{eq3.3}, we have
\begin{equation}\label{eq3.10}
(\alpha+1)\int_{0}^{2\pi}\frac{d\theta}{2\pi}\ind\langle g,
\,U^{\alpha}_{z}\phi_\theta\rangle\langle U^{\alpha}_{z}
\psi_\theta,\,h \rangle\,d\lambda(z)=\langle g,\,h
\rangle\langle \psi,\,\phi\rangle
\end{equation}
for all $\phi,\psi,g,h\in\mathcal D_{\alpha}$.

Fix $\phi,\psi\in\mathcal D_{\alpha}$. The sesquilinear form
\begin{equation}\label{eq3.11}
(g,h)\mapsto(\alpha+1)\int_{0}^{2\pi}\frac{d\theta}{2\pi}
\ind\langle g,\,U^{\alpha}_{z}\phi_\theta\rangle\langle
U^{\alpha}_{z}\psi_\theta,\,h \rangle\,d\lambda(z)
\end{equation}
is well defined on $\mathcal D_{\alpha}\times
\mathcal D_{\alpha}$ and coincides with the restriction to
$\mathcal D_{\alpha}\times\mathcal D_{\alpha}$
of the sesquilinear form
$$(g,h)\mapsto\langle g,\,h\rangle\langle \psi,\,\phi
\rangle,\qquad g,h\in A_{\alpha}^2.$$
Thus \eqref{eq3.11} extends to a bounded sesquilinear form on
$A_{\alpha}^2\times A_{\alpha}^2$. As a result, the
identity \eqref{eq3.10} holds for $g,h\in A_{\alpha}^2$ and
$\phi,\psi\in \mathcal D_{\alpha}$. Similarly, for any
$g,h\in A_{\alpha}^2$, the sesquilinear form
$$(\psi,\phi)\mapsto (\alpha+1)\int_{0}^{2\pi}
\frac{d\theta}{2\pi}\ind\langle g,\,U^{\alpha}_{z}
\phi_\theta\rangle\langle U^{\alpha}_{z}\psi_\theta,\,h
\rangle\,d\lambda(z)$$
coincides with $(\psi,\phi)\mapsto \langle g,\,h\rangle
\langle \psi,\,\phi\rangle$ and extends boundedly to
$A_{\alpha}^2\times A_{\alpha}^2$. This proves
the desired result.
\end{proof}

The orthogonality relation in Theorem~\ref{3.4} explains why
localization operators for $A^2_\alpha$ are defined as in
Definition~\ref{Bergman localization}. By Theorem \ref{3.4} and
the Cauchy-Schwarz inequality, we see that
$$\|\mathbf L^{\phi,\psi,\alpha}_f\|\leq \|f\|_{\infty}$$
for $f\in L^{\infty}(\T\times \D)$ and $\phi,\psi\in A_{\alpha}^2$.

Clearly, the localization operator $\mathbf L_{f}^{\phi,\psi,\alpha}$ on
$A_{\alpha}^2$ given in Definition~\ref{Bergman localization} also
makes sense for $f$ in certain symbol classes more general than
$L^\infty(\T\times\D)$. For example, if $f\in L^1(\T\times\D,dH)$,
then $\mathbf L^{\phi,\psi,\alpha}_f$ is well-defined. Note that
$\mathbf L^{\phi,\psi,\alpha}_f$ can be rewritten as the following
operator-valued integral:
\begin{equation}\label{eq3.12}
\mathbf L^{\phi,\psi,\alpha}_f=(\alpha+1)\int_0^{2\pi}
\frac{d\theta}{2\pi}\ind f(e^{i\theta},z)
\big[(U^{\alpha}_{z}\psi_\theta)\otimes
(U^{\alpha}_{z}\phi_\theta)\big]\,d\lambda(z),
\end{equation}
where $x\otimes y$ is the rank-one operator on $A^2_\alpha$
defined by $(x\otimes y)z=\langle z,\,y\rangle x$. The same remark
applies to localization operators on the Fock space $F^2_\beta$.

\section{Weak convergence of localization operators on $A^2_{\alpha}$}

In this section, we show that localization operators on the Bergman
space $A_{\beta r^2}^2$ converge weakly to a localization operator
on the Fock space $F_{\beta}^2$ as $r$ goes to infinity. We also
give several applications of this limit theorem.

\begin{lem}\label{4.1}
For any $\beta>0$ and $\phi,\psi,g,h\in F_{\beta}^2$, we have
\begin{align*}
\lim_{r\to \infty}(\beta r^2+1)&\int_{0}^{2\pi}
\frac{d\theta}{2\pi}\ind\langle g_r,\,U^{\beta r^2}_z
(\phi_r)_\theta\rangle_{A_{\beta r^2}^2}
\langle U^{\beta r^2}_z(\psi_r)_\theta,\,h_r
\rangle_{A_{\beta r^2}^2}\,d\lambda(z)\\
&=\frac{\beta}{\pi}\inc\langle g,\,W^{\beta}_z\phi
\rangle_{F_{\beta}^2}\langle W^{\beta}_z\psi,\,h
\rangle_{F_{\beta}^2}\,dA(z)\\
&=\frac{\beta}{\pi}\int_0^{2\pi}\frac{d\theta}{2\pi}
\inc\langle g,\,W^{\beta}_z\phi_{\theta}
\rangle_{F_{\beta}^2}\langle W^{\beta}_z\psi_{\theta},
\,h\rangle_{F_{\beta}^2}\,dA(z),
\end{align*}
where $\phi_{\theta}(\zeta)=\phi(e^{i\theta}\zeta)$ and
$g_r(\zeta)=g(r\zeta)$.
\end{lem}

\begin{proof}
From Theorem \ref{2.1}, \ref{3.4}, and \ref{3.1},
we deduce that
\begin{align*}
(\beta r^2+1)&\int_{0}^{2\pi}\frac{d\theta}{2\pi}\ind\langle g_r,
\,U^{\beta r^2}_z(\phi_r)_\theta\rangle_{A_{\beta r^2}^2}
\langle U^{\beta r^2}_z(\psi_r)_\theta,\,h_r
\rangle_{A_{\beta r^2}^2}\,d\lambda(z)\\
&=\langle g_r,\,h_r\rangle_{A_{\beta r^2}^2}\langle\psi_r,\,\phi_r
\rangle_{A_{\beta r^2}^2}\\
&=\left(\ind g(rz)\overline{h(rz)}\,dA_{\beta r^2}(z)\right)
\left(\ind\psi(rz)\overline{\phi(rz)}\,dA_{\beta r^2}(z)\right)\\[8pt]
&\stackrel{r\to \infty}{\xrightarrow{\hspace*{1cm}}}
\left(\inc g(z)\overline{h(z)}\,d\mu_{\beta}(z)\right)
\left(\inc\psi(z)\overline{\phi(z)}\,d\mu_{\beta}(z)\right)\\
&=\langle g,\,h\rangle_{F_{\beta}^2}\langle\psi,\,\phi
\rangle_{F_{\beta}^2}\\
&=\frac{\beta}{\pi}\inc\langle g,
\,W^{\beta}_z\phi\rangle_{F_{\beta}^2}\langle
W_z^{\beta}\psi,\,h\rangle_{F_{\beta}^2}\,dA(z)\\
&=\frac{\beta}{\pi}\int_0^{2\pi}\frac{d\theta}{2\pi}
\inc\langle g,\,W^{\beta}_z\phi_{\theta}
\rangle_{F_{\beta}^2}\langle W^{\beta}_z
\psi_{\theta},\,h\rangle_{F_{\beta}^2}\,dA(z),
\end{align*}
where the last equality follows from the rotation-invariance
of the Gaussian measure; see the early part of the previous
section. This completes the proof of the lemma.
\end{proof}

\begin{lem}\label{4.2}
Let $\phi,\psi\in F_{\beta}^2$ with $\beta>0$. For
$g,h\in F_{\gamma}^2$ with $0<\gamma<\beta$,
$z\in\C$ and $\theta\in [0,2\pi)$, we have
$$\lim_{r\to \infty}\frac{\langle g_r,\,U^{\beta r^2}_{z/r}
(\phi_r)_\theta\rangle_{A_{\beta r^2}^2}\langle
U^{\beta r^2}_{z/r}(\psi_r)_\theta,\,h_r
\rangle_{A_{\beta r^2}^2}}{\left(1-\left|\frac{z}{r}
\right|^2\right)^2}=\langle g,\,W^{\beta}_z\phi_{\theta}
\rangle_{F_{\beta}^2}\langle W^{\beta}_z
\psi_{\theta},\,h\rangle_{F_{\beta}^2}.$$
\end{lem}

\begin{proof}
It suffices to show that
\begin{equation}\label{eq4.1}
\lim_{r\to \infty}\frac{\langle g_r,\,U^{\beta r^2}_{z/r}
(\phi_r)_\theta\rangle_{A_{\beta r^2}^2}}{1-\left|
\frac{z}{r}\right|^2}=\langle g,\,W^{\beta}_z
\phi_{\theta}\rangle_{F_{\beta}^2}.
\end{equation}
Fix $z\in\C$ and make a change of variables. For $|z|<r$,
we obtain{\small
\begin{align}
&\frac{\langle g_r,\,U^{\beta r^2}_{z/r}(\phi_r)_\theta
\rangle_{A_{\beta r^2}^2}}{1-\left|\frac{z}{r}
\right|^2}\nonumber\\
&=\frac{\beta r^2+1}{1-\left|\frac{z}{r}\right|^2}
\ind g_r(\zeta)\overline{\phi_r\left(e^{i\theta}
\frac{\zeta-\frac{z}{r}}{1-\frac{\overline z\zeta}{r}}
\right)\left[\frac{1-\left|\frac{z}{r}\right|^2}
{\left(1-\frac{\overline z\zeta}{r}\right)^2}\right]^{1+
\frac{\beta r^2}{2}}}(1-|\zeta|^2)^{\beta r^2}
\,\frac{dA(\zeta)}{\pi}\nonumber\\
&=\frac{\beta r^2+1}{1-\left|\frac{z}{r}\right|^2}\ind
g(r\zeta)\overline{\phi\left(e^{i\theta}\frac{r\zeta- z}
{1-\frac{\overline z\zeta}{r}}\right)\left[\frac{1-\left|\frac{z}{r}
\right|^2}{\left(1-\frac{\overline z\zeta}{r}\right)^2}
\right]^{1+\frac{\beta r^2}{2}}}(1-|\zeta|^2)^{\beta r^2}
\,\frac{dA(\zeta)}{\pi}\nonumber\\
&=\frac{\beta r^2+1}{r^2 }\left(1-\left|\frac{z}{r}\right|^2
\right)^{\frac{\beta r^2}{2}}\int_{r\D}g(\zeta)
\overline{\phi\left(e^{i\theta}\frac{\zeta-z}
{1-\frac{\overline z\zeta}{r^2}}\right)\left(1-\frac{\overline z\zeta}
{r^2}\right)^{-(2+\beta r^2)}}\left(1-\frac{|\zeta|^2}{r^2}\right)^{\beta r^2}\!\frac{dA(\zeta)}{\pi}\nonumber\\
&=\frac{\beta r^2+1}{r^2 }\left(1-\left|\frac{z}{r}\right|^2
\right)^{\frac{\beta r^2}{2}}\inc g(\zeta)\overline{\phi\left(
e^{i\theta}\frac{\zeta-z}{1-\frac{\overline z\zeta}{r^2}}\right)}
F_r(\zeta)\,\frac{dA(\zeta)}{\pi},\label{eq4.2}
\end{align}
}where
$$F_r(\zeta)=\mathbbm{1}_{r\D}(\zeta)
\left(1-\frac{z\overline\zeta}{r^2}\right)^{-(2+\beta r^2)}
\left(1-\frac{|\zeta|^2}{r^2}\right)^{\beta r^2}.$$
By \eqref{eq2.3}, we have
\begin{equation}\label{eq4.3}
\mathbbm{1}_{r\D}(\zeta)\left(1-\frac{|\zeta|^2}{r^2}
\right)^{\beta r^2}=\mathbbm{1}_{r\D}(\zeta)
\left[\left(1-\frac{|\zeta|^2}{r^2}
\right)^{\frac{\beta r^2}{|\zeta|^2}}\right]^{|\zeta|^2}
\leq e^{-\beta|\zeta|^2}.
\end{equation}
It is easy to see that
\begin{equation}\label{eq4.4}
\left(1-\frac{1}{x}\right)^{1-x}=\left( 1+\frac{1}{x-1}
\right)^{x-1}\leq e,\qquad x\in(1,\infty).
\end{equation}
Since $\left|1- \frac{\overline z \zeta}{r^2}\right|\geq
1-\frac{|z\zeta|}{r^2}$, it follows from \eqref{eq4.4} that
\begin{align}
\mathbbm{1}_{r\D}(\zeta)\left|1- \frac{ z\overline\zeta}{r^2}
\right|^{-\left(2+\beta r^2\right)}
&\leq\mathbbm{1}_{r\D}
(\zeta)\left( 1-\frac{|z\zeta|}{r^2}\right)^{-\left(2+\beta r^2
\right)}\nonumber\\
&=\mathbbm{1}_{r\D}(\zeta)\left[\left( 1-\frac{|z\zeta|}
{r^2}\right)^{1-\frac{r^2}{|z\zeta|}}\right]^{\left(2+\beta r^2
\right)\frac{|z\zeta|}{r^2-|z\zeta|}}\nonumber\\
&\leq \mathbbm{1}_{r\D}(\zeta)\exp\left[\frac{(2+\beta r^2)
|z\zeta|}{r^2-|z\zeta|}\right]\nonumber\\
&=\mathbbm{1}_{r\D}(\zeta)\exp\left[\frac{\beta r^2|z\zeta|}
{r^2-|z\zeta|}+\frac{2|z\zeta|}{r^2-|z\zeta|}\right].\label{eq4.5}
\end{align}
Since $\gamma<\beta$, we can choose a positive integer $n$
such that $\gamma<\beta\frac{(n-2)^2-2}{(n-1)^2}$. We may
assume $r$ is large enough such that $|z|<\frac{r}{n}$. Then
for $|\zeta|<r$ we have
$$r^2-|z\zeta|\geq r^2-r|z|\geq r^2-\frac{r^2}{n}
=\frac{n-1}{n}r^2,$$
and thus
$$\frac{2|z\zeta|}{r^2-|z\zeta|}\leq \frac{2n}{n-1}\frac{|z|}{r}
\frac{|\zeta|}{r}\leq \frac{2}{n-1},\qquad\frac{\beta r^2|z\zeta|}
{r^2-|z\zeta|}\leq \frac{\beta n}{n-1}|z\zeta|.$$
Combining this with \eqref{eq4.5}, we obtain
\begin{equation}\label{eq4.6}
\mathbbm{1}_{r\D}(\zeta)\left|1- \frac{\overline z \zeta}{r^2}
\right|^{-(2+\beta r^2)}\leq e^{\frac{2}{n-1}}
e^{\frac{\beta n}{n-1}|z\zeta|}.
\end{equation}
By \eqref{eq4.3} and \eqref{eq4.6},
$$|F_r(\zeta)|\leq e^{\frac{2}{n-1}}e^{\frac{\beta n}{n-1}
|z\zeta|}e^{-\beta |\zeta|^2}.$$
Since $g\in F_{\gamma}^2$ and $\phi\in F_{\beta}^2$, we have
$$|g(\zeta)|\leq \|g\|_{F_{\gamma}^2}
e^{\frac{\gamma|\zeta|^2}{2}}$$
for $|\zeta|<r$, and
\begin{align*}
\left|\phi\left(e^{i\theta}\frac{\zeta-z}{1-\frac{\overline z\zeta}
{r^2}}\right)\right|\leq\|\phi\|_{F_{\beta}^2}\exp
\left[\frac{\beta\big|\frac{\zeta-z}{1- \frac{\overline z\zeta}{r^2}}
\big|^2}{2}\right]\leq \|\phi\|_{F_{\beta}^2}\exp
\left[\frac{\beta n^2}{2(n-1)^2}|\zeta-z|^2\right].
\end{align*}
It follows that
$$\left|g(\zeta)\overline{\phi\left(e^{i\theta}\frac{\zeta-z}
{1-\frac{\overline z\zeta}{r^2}}\right)}F_r(\zeta)\right|\leq
\big(e^{\frac{2}{n-1}}\|\phi\|_{F_{\beta}^2}\|g\|_{F_{\gamma}^2}
\big)e^{\frac{\beta n}{n-1}|z\zeta|+\frac{\beta n^2}{2(n-1)^2}
|\zeta-z|^2-(\beta-\frac{\gamma}{2})|\zeta|^2}.$$
Since $\gamma<\beta\frac{(n-2)^2-2}{(n-1)^2}$, we have
$\frac{\beta n^2}{2(n-1)^2}<\beta-\frac{\gamma}{2}$ and thus
$$\inc e^{\frac{\beta n}{n-1}|z\zeta|+\frac{\beta n^2}{2(n-1)^2}
|\zeta-z|^2-(\beta-\frac{\gamma}{2})|\zeta|^2}
\,dA(\zeta)<\infty.$$
Applying the dominated convergence theorem, we obtain
\begin{align}\label{eq4.7}
&\lim_{r\to\infty}\int_{\C}g(\zeta)\overline{\phi\left(e^{i\theta}
\frac{\zeta-z}{1- \frac{\overline z\zeta}{r^2}}\right)}
F_r(\zeta)\,\frac{dA(\zeta)}{\pi}\nonumber\\
&\qquad\qquad=\inc g(\zeta)\overline{\phi(e^{i\theta}(\zeta-z))
e^{\beta \overline z\zeta}}e^{-\beta|\zeta|^2}\,\frac{dA(\zeta)}{\pi}.
\end{align}
The desired equality \eqref{eq4.1} now follows
from \eqref{eq4.2} and \eqref{eq4.7}.
\end{proof}

We now prove the main result of this section.

\begin{thm}\label{4.3}
Suppose $\phi,\psi\in F_{\beta}^2$ with $\beta>0$ and
$f\in L^{\infty}(\T\times\C)$. For any $\sigma\geq 0$, we have
\begin{equation}\label{eq4.8}
\lim_{r\to \infty}\langle \mathbf L_{f_{r,\sigma}}^{\phi_r,\psi_r,\beta r^2}
g_r,\,h_r\rangle_{A_{\beta r^2}^2}=\langle \mathbb L_f^{\phi,\psi,\beta}g,
\,h\rangle_{F_{\beta}^2}, \qquad g,h\in F^2_{\beta},
\end{equation}
where
$$f_{r,\sigma}(e^{i\theta},z)=(1-|z|^2)^{\sigma}f(e^{i\theta},rz).$$
\end{thm}

\begin{proof}
First assume $g,h\in F_{\gamma}^2$ with $0<\gamma<\beta$.
By a change of variables,{\footnotesize
\begin{align}
&\langle \mathbf L_{f_{r,\sigma}}^{\phi_r,\psi_r,\beta r^2} g_r,
\,h_r\rangle_{A_{\beta r^2}^2}\nonumber\\
&=(\beta r^2+1)\int_{0}^{2\pi}\frac{d\theta}{2\pi}
\ind f_{r,\sigma}(e^{i\theta},z)\langle g_r,\,
U_{z}^{\beta r^2}(\phi_r)_{\theta}\rangle_{A_{\beta r^2}^2}
\langle U_{z}^{\beta r^2}(\psi_r)_{\theta},\,h_r
\rangle_{A_{\beta r^2}^2}\,d\lambda(z)\nonumber\\
&=\frac{\beta r^2+1}{\pi r^2}\int_{0}^{2\pi}\frac{d\theta}{2\pi}
\int_{r\D}\left(1-\left|\frac{z}{r}\right|^2\right)^{\sigma}
f(e^{i\theta},z)\frac{\langle g_r,\,U_{z/r}^{\beta r^2}
(\phi_r)_\theta\rangle_{A_{\beta r^2}^2}\langle
U_{z/r}^{\beta r^2}(\psi_r)_\theta,\,h_r
\rangle_{A_{\beta r^2}^2}}{\left(1-\left|\frac{z}{r}
\right|^2\right)^2}\,dA(z)\nonumber\\
&=\int_{0}^{2\pi}\frac{d\theta}{2\pi}
\inc F_r(z)\,dA(z). \label{eq4.9}
\end{align}
}where{\small
$$F_r(z)=\frac{\beta r^2+1}{\pi r^2}\mathbbm{1}_{r\D}(z)
\left(1-\left|\frac{z}{r}\right|^2\right)^{\sigma}f(e^{i\theta},z)
\frac{\langle g_r,\,U_{z/r}^{\beta r^2}(\phi_r)_\theta
\rangle_{A_{\beta r^2}^2}\langle U_{z/r}^{\beta r^2}(\psi_r)_\theta,
\,h_r\rangle_{A_{\beta r^2}^2}}
{\left(1-\left|\frac{z}{r}\right|^2\right)^2}.$$
}By Lemma \ref{4.2}, we have
\begin{equation}\label{eq4.10}
\lim_{r\to \infty}F_r(z)=\frac{\beta}{\pi}f(e^{i\theta},z)
\langle g,\,W^{\beta}_z\phi_{\theta}\rangle_{F_{\beta}^2}
\langle W_z^{\beta}\psi_{\theta},\,h\rangle_{F_{\beta}^2}
\end{equation}
for almost every $z\in\C$ and $e^{i\theta}\in\T$. We also have
\begin{align}
|F_r(z)|&\leq \|f\|_{\infty}\frac{\beta r^2+1}{\pi r^2}
\mathbbm{1}_{r\D}(z)\frac{|\langle g_r,\,U_{z/r}^{\beta r^2}
(\phi_r)_\theta\rangle_{A_{\beta r^2}^2}|^2}{2\left(1-\left|
\frac{z}{r}\right|^2\right)^2}\nonumber\\
&\qquad+\|f\|_{\infty}\frac{\beta r^2+1}{\pi r^2}
\mathbbm{1}_{r\D}(z)\frac{|\langle U_{z/r}^{\beta r^2}
(\psi_r)_\theta,\,h_r\rangle_{A_{\beta r^2}^2}|^2}
{2\left(1-\left|\frac{z}{r}\right|^2\right)^2}.\label{eq4.11}
\end{align}
By Lemma \ref{4.2} again, for almost every $z\in\C$ and
$e^{i\theta}\in\T$, we have
\begin{equation}\label{eq4.12}
\lim_{r\to \infty}\frac{\beta r^2+1}{\pi r^2}\mathbbm{1}_{r\D}(z)
\frac{|\langle g_r,\,U_{z/r}^{\beta r^2}(\phi_r)_\theta
\rangle_{A_{\beta r^2}^2}|^2}{2\left(1-\left|\frac{z}{r}\right|^2
\right)^2}=\frac{\beta}{2\pi}|\langle g,\,W^{\beta}_z
\phi_{\theta}\rangle_{F_{\beta}^2}|^2,
\end{equation}
and
\begin{equation}\label{eq4.13}
\lim_{r\to \infty}\frac{\beta r^2+1}{\pi r^2}\mathbbm{1}_{r\D}(z)
\frac{|\langle U_{z/r}^{\beta r^2}(\psi_r)_\theta,\,h_r
\rangle_{A_{\beta r^2}^2}|^2}{2\left(1-\left|\frac{z}{r}\right|^2
\right)^2}=\frac{\beta}{2\pi}|\langle W^{\beta}_z
\psi_{\theta},\,h\rangle_{F_{\beta}^2}|^2.
\end{equation}
By a change of variables and Lemma \ref{4.1}, we have
\begin{align}
&\lim_{r\to \infty}\int_{0}^{2\pi}\frac{d\theta}{2\pi}
\inc\frac{\beta r^2+1}{\pi r^2}\mathbbm{1}_{r\D}(z)
\frac{|\langle g_r,\,U_{z/r}^{\beta r^2}
(\phi_r)_\theta\rangle_{A_{\beta r^2}^2}|^2}{2\left(1-\left|
\frac{z}{r}\right|^2\right)^2}\,dA(z)\nonumber\\
&\quad=\lim_{r\to \infty}\frac{\beta r^2+1}{2}
\int_{0}^{2\pi}\frac{d\theta}{2\pi}\ind|\langle g_r,\,
U_{z}^{\beta r^2}(\phi_r)_\theta
\rangle_{A_{\beta r^2}^2}|^2d\lambda(z)\nonumber\\
&\quad=\frac{\beta}{2\pi}\int_{0}^{2\pi}\frac{d\theta}{2\pi}
\inc|\langle g,\,W^{\beta}_z\phi_{\theta}
\rangle_{F_{\beta}^2}|^2\,dA(z).\label{eq4.14}
\end{align}
and
\begin{align}
&\lim_{r\to\infty}\int_{0}^{2\pi}\frac{d\theta}{2\pi}
\inc\frac{\beta r^2+1}{\pi r^2}
\mathbbm{1}_{r\D}(z)\frac{|\langle U_{z/r}^{\beta r^2}
(\psi_r)_\theta,\,h_r\rangle_{A_{\beta r^2}^2}|^2}
{2\left(1-\left|\frac{z}{r}\right|^2\right)^2}\,dA(z)\nonumber\\
&\quad=\frac{\beta}{2\pi}\int_{0}^{2\pi}\frac{d\theta}{2\pi}
\inc|\langle W^{\beta}_z\psi_{\theta},\,h
\rangle_{F_{\beta}^2}|^2\,dA(z).\label{eq4.15}
\end{align}
In view of \eqref{eq4.9}, \eqref{eq4.10}, \eqref{eq4.11},
\eqref{eq4.12}, \eqref{eq4.13}, \eqref{eq4.14} and \eqref{eq4.15},
we can apply the dominated convergence theorem to get{\small
\begin{align*}
\lim_{r\to \infty}\langle
\mathbf L_{f_{r,\sigma}}^{\phi_r,\psi_r,\beta r^2} g,\,h
\rangle_{A_{\beta r^2}^2}&=\frac{\beta}{\pi}\int_{0}^{2\pi}
\frac{d\theta}{2\pi}\inc f(e^{i\theta},z)\langle g,\,
W^{\beta}_z\phi_{\theta}\rangle_{F_{\beta}^2}
\langle W^{\beta}_z\psi_{\theta},\,h\rangle_{F_{\beta}^2}\,dA(z)\\
&=\langle \mathbb L_f^{\phi,\psi,\beta}g,\,h\rangle_{F_{\beta}^2}.
\end{align*}
}This shows that \eqref{eq4.8} holds for all $g,h\in F_{\gamma}^2$
with $\gamma<\beta$.

In general, let $g,h\in F_{\beta}^2$. For any $\varepsilon>0$,
there are $g_{\varepsilon}$ and $h_{\varepsilon}$ in
$F_{\gamma}^2$ with $\gamma<\beta$ such that
$$\|g-g_{\varepsilon}\|_{F_{\beta}^2}<\varepsilon,
\qquad \|h-h_{\varepsilon}\|_{F_{\beta}^2}<\varepsilon.$$
By the triangle inequality,
\begin{align*}
&\left|\langle \mathbb L_f^{\phi,\psi,\beta}g,\,h
\rangle_{F_{\beta}^2}-\langle \mathbf L_{f_{r,\sigma}}^{\phi_r,\psi_r,
\beta r^2} g_r,\,h_r\rangle_{A_{\beta r^2}^2}\right|
\leq\left|\langle \mathbb L_f^{\phi,\psi,\beta}g,\,h
\rangle_{F_{\beta}^2}-\langle \mathbb L_f^{\phi,\psi,\beta}
g_{\varepsilon},\,h\rangle_{F_{\beta}^2}\right|\\
&\qquad+\left|\langle \mathbb L_f^{\phi,\psi,\beta}g_{\varepsilon},
\,h\rangle_{F_{\beta}^2}-\langle \mathbb L_f^{\phi,\psi,\beta}
g_{\varepsilon},\,h_{\varepsilon}\rangle_{F_{\beta}^2}\right|\\
&\qquad +\left|\langle \mathbb L_f^{\phi,\psi,\beta}g_{\varepsilon},
\,h_{\varepsilon}\rangle_{F_{\beta}^2}-\langle
\mathbf L_{f_{r,\sigma}}^{\phi_r,\psi_r,\beta r^2} (g_{\varepsilon})_r,\,(h_{\varepsilon})_r\rangle_{A_{\beta r^2}^2}\right|\\
&\qquad+\left|\langle \mathbf L_{f_{r,\sigma}}^{\phi_r,\psi_r,\beta r^2}
(g_{\varepsilon})_r,\,(h_{\varepsilon})_r\rangle_{A_{\beta r^2}^2}-
\langle \mathbf L_{f_{r,\sigma}}^{\phi_r,\psi_r,\beta r^2} g_r,
\,(h_{\varepsilon})_r\rangle_{A_{\beta r^2}^2}\right|\\
&\qquad+\left|\langle \mathbf L_{f_{r,\sigma}}^{\phi_r,\psi_r,\beta r^2}
g_r,\,(h_{\varepsilon})_r\rangle_{A_{\beta r^2}^{2}}-\langle
\mathbf L_{f_{r,\sigma}}^{\phi_r,\psi_r,\beta r^2} g_r,\,h_r
\rangle_{A_{\beta r^2}^{2}}\right|
\end{align*}
Since $g_{\varepsilon},h_{\varepsilon}\in F_{\gamma}^2$,
our early analysis shows that
$$\lim_{r\to \infty}\langle \mathbf L_{f_{r,\sigma}}^{\phi_r,\psi_r,\beta r^2}
(g_{\varepsilon})_r,\,(h_{\varepsilon})_r\rangle_{A_{\beta r^2}^2}=
\langle \mathbb L_f^{\phi,\psi,\beta}g_{\varepsilon},
\,h_{\varepsilon}\rangle_{F_{\beta}^2}.$$
Moreover, by the Cauchy-Schwarz inequality, we have
$$\left|\langle \mathbb L_f^{\phi,\psi,\beta}g,\,h
\rangle_{F_{\beta}^2}-\langle \mathbb L_f^{\phi,\psi,\beta}
g_{\varepsilon},\,h\rangle_{F_{\beta}^2}\right|=\left|\langle
\mathbb L_f^{\phi,\psi,\beta}(g-g_{\varepsilon}),\,h
\rangle_{F_{\beta}^2}\right|\leq \|f\|_{\infty}\|h\|\varepsilon,$$
and
$$\left|\langle \mathbb L_f^{\phi,\psi,\beta}g_{\varepsilon},\,h
\rangle_{F_{\beta}^2}-\langle \mathbb L_f^{\phi,\psi,\beta}
g_{\varepsilon},\,h_{\varepsilon}\rangle_{F_{\beta}^2}\right|
\leq \|f\|_{\infty}\|g_{\varepsilon}\|\varepsilon.$$
By Corollary \ref{2.2} and the Cauchy-Schwarz inequality again,
\begin{align*}
&\left|\langle \mathbf L_{f_{r,\sigma}}^{\phi_r,\psi_r,\beta r^2}
(g_{\varepsilon})_r,\,(h_{\varepsilon})_r\rangle_{A_{\beta r^2}^2}-
\langle \mathbf L_{f_{r,\sigma}}^{\phi_r,\psi_r,\beta r^2} g_r,
\,(h_{\varepsilon})_r\rangle_{A_{\beta r^2}^2}\right|\\
&\qquad\leq \|f\|_{\infty}\|(h_{\varepsilon})_r\|_{A_{\beta r^2}^2}
\|(g-g_{\varepsilon})_r\|_{A_{\beta r^2}^2}\\
&\qquad\rightarrow\|f\|_{\infty}\|h_{\varepsilon}\|_{F_{\beta}^2}
\|g-g_{\varepsilon}\|_{F_{\beta}^2}
\end{align*}
as $r\to\infty$, and
\begin{align*}
&\left|\langle \mathbf L_{f_{r,\sigma}}^{\phi_r,\psi_r,\beta r^2} g_r,
\,(h_{\varepsilon})_r\rangle_{A_{\beta r^2}^2}-\langle
\mathbf L_{f_{r,\sigma}}^{\phi_r,\psi_r,\beta r^2} g_r,\,h_r
\rangle_{A_{\beta r^2}^2}\right|\\
&\qquad\leq \|f\|_{\infty}\|g_r\|_{A_{\beta r^2}^2}\|
(h-h_{\varepsilon})_r\|_{A_{\beta r^2}^2}\\
&\qquad\rightarrow\|f\|_{\infty}\|g\|_{F_{\beta}^2}
\|h-h_{\varepsilon}\|_{F_{\beta}^2}
\end{align*}
as $r\to\infty$.
This shows that \eqref{eq4.8} holds for all $g,h\in F_{\beta}^2$.
\end{proof}

\begin{cor}\label{4.4}
Let $\phi,\psi\in F_{\beta}^2$ with $\beta>0$ and
$f\in L^{\infty}(\T\times\C)$. For any $\sigma\geq0$ we have
\begin{equation}\label{eq4.16}
\|\mathbb L_f^{\phi,\psi,\beta}\|_{F_{\beta}^2}\leq\limsup_{r\to\infty}
\|\mathbf L_{f_{r,\sigma}}^{\phi_r,\psi_r,\beta r^2}\|_{A_{\beta r^2}^2},
\end{equation}
where $f_{r,\sigma}(e^{i\theta},z)=(1-|z|^2)^{\sigma}f(e^{i\theta},rz)$.
\end{cor}

\begin{proof}
By Theorem \ref{4.3}, Corollary \ref{2.2}, and the Cauchy-Schwarz
inequality, we have
\begin{align*}
|\langle \mathbb L_f^{\phi,\psi,\beta }g,\,h\rangle_{F_{\beta}^2}|
&=\lim_{r\to \infty}|\langle \mathbf L_{f_{r,\sigma}}^{\phi_r,\psi_r,\beta r^2}
g_r,\,h_r\rangle_{A_{\beta r^2}^2}|\\
&\leq\limsup_{r\to \infty}\big(\|\mathbf L_{f_{r,\sigma}}^{\phi_r,\psi_r,
\beta r^2}\|_{A_{\beta r^2}^2}\|g_r\|_{A_{\beta r^2}^2}
\|h_r\|_{A_{\beta r^2}^2}\big)\\
&=\big(\limsup_{r\to \infty}\|\mathbf L_{f_{r,\sigma}}^{\phi_r,\psi_r,r^2}
\|_{A_{\beta r^2}^2}\big)\|g\|_{F_{\beta}^2}\|h\|_{F_{\beta}^2}
\end{align*}
for all $g,h\in F_{\beta}^2$. Taking the supremum over all $h\in
F_{\beta}^2$ with $\|h\|_{F^2_\beta}\le1$ yields the desired inequality.
\end{proof}

If $\phi=\psi=1$, Theorem \ref{4.3} and Corollary \ref{4.4} establish
a connection between Toeplitz operators on Fock spaces and those on
weighted Bergman spaces. Also note that \eqref{eq4.16} holds for more
general symbols and the inequality is sharp. For
$f\in L^2(\C,d\mu_{\beta})$ such that $fK_{z}^{\beta}\in
L^2(\C,d\mu_{\beta})$ for all $z\in\C$, the Toeplitz operator
$\mathbb T_f$ is densely defined on $F_{\beta}^2$ by
$$\mathbb T^{\beta}_fg(z)=\int_{\C}f(w)g(w)e^{\beta z\bar w}
\,d\mu_{\beta}(w),$$
where $g\in \mathcal D_{F_{\beta}^2}=
{\rm span}\{K^{\beta}_z:z\in\C\}$.

\begin{thm}\label{4.5}
Let $\beta>0$. If $f\in L^2(\C,d\mu_{\beta})$ such that
$fK_{z}^{\beta}\in L^2(\C,d\mu_{\beta})$ for all
$z\in\C$, then
\begin{equation}\label{eq4.17}
\|\mathbb T^{\beta}_f\|_{F_{\beta}^2}\leq\limsup_{r\to\infty}
\|\mathbf T^{\beta r^2}_{f_{r,\sigma}}\|_{A_{\beta r^2}^2}
\end{equation}
for any $\sigma\geq 0$, where $f_{r,\sigma}(z)=
(1-|z|^2)^{\sigma}f(rz)$. Moreover, equality holds if
$f=\mathbbm{1}_{\Omega}$, where $\Omega$ is any Euclidean
disc in $\C$ centered at $0$, and $\mathbbm{1}_{\Omega}$ is
the characteristic function of $\Omega$.
\end{thm}

\begin{proof}
By Theorem \ref{2.1} and Corollary \ref{2.2}, we have
\begin{align*}
|\langle \mathbb T^{\beta}_fg,\,h\rangle_{F_{\beta}^2}|
&=\left|\inc f(w)g(w)\overline{h(w)}d\mu_{\beta}(w)\right|\\
&=\lim_{r\to \infty}\left|\ind f_{r}(w)g_r(w)\overline{h_r(w)}
\,dA_{\beta r^2+\sigma}(w)\right|\\
&=\lim_{r\to \infty}\frac{\beta r^2+\sigma+1}{\beta r^2+1}
\left|\ind f_{r,\sigma}(w)g_r(w)\overline{h_r(w)}
\,dA_{\beta r^2}(w)\right|\\
&=\lim_{r\to \infty}\left|\langle\mathbf T^{\beta r^2}_{f_{r,\sigma}}
g_r,\,h_r\rangle_{A_{\beta r^2}^2}\right|\\
&\leq\limsup_{r\to \infty}\left(\|\mathbf T^{\beta r^2}_{f_{r,\sigma}}
\|_{A_{\beta r^2}^2}\|g_r\|_{A_{\beta r^2}^2}
\|h_r\|_{A_{\beta r^2}^2}\right)\\
&=\left(\limsup_{r\to \infty}\|\mathbf T^{\beta r^2}_{f_{r,\sigma}}
\|_{A_{\beta r^2}^2}\right)\|g\|_{F_{\beta}^2}
\|h\|_{F_{\beta}^2}
\end{align*}
for $g,h\in \mathcal D_{F_{\beta}^2}$. It follows that
$$\|\mathbb T^{\beta}_f\|_{F_{\beta}^2}\leq
\limsup_{r\to\infty}\|\mathbf T^{\beta r^2}_{f_{r,\sigma}}
\|_{A_{\beta r^2}^2}.$$

Next, we show that equality holds in \eqref{eq4.17} if
$f=\mathbbm{1}_{\Omega}$, where $\Omega$ is the
disc $B(0,R)$ in $\C$ centered at $0$ with radius $R$.
An orthonormal basis of $F_{\beta}^2$ is given by
$$\omega^{\beta}_n(z)=\sqrt{\frac{\beta^n}{n!}}
\,z^n,\qquad n\ge0.$$
For non-negative integers $m$ and $n$, we have
\begin{align*}
\langle\mathbb T_f^{\beta}\omega^{\beta}_n,\,
\omega^{\beta}_m\rangle_{F_{\beta}^2}
&=\sqrt{\frac{\beta^n\beta^m}{n!m!}}\int_{B(0,R)}
\zeta^n\bar{\zeta}^m\,d\mu_{\beta}(\zeta)\\
&=\frac{\beta}{\pi}\sqrt{\frac{\beta^n\beta^m}{n!m!}}
\int_0^R\int_0^{2\pi}x^{n+m+1}e^{i(n-m)\theta}
e^{-\beta x^2}\,d\theta\,dx\\
&=\begin{cases}
0, & n\neq m,\\\noalign{\vskip 8pt}
\displaystyle\frac{1}{n!}\int_0^{\beta R^2}
x^{n}e^{-x}\,dx,  &n=m.
\end{cases}
\end{align*}
Thus $\mathbb T^{\beta}_f$ is a diagonal operator with
$$\mathbb T^{\beta}_f\omega^{\beta}_n=
\gamma_n\omega^{\beta}_n,\quad n=0,1,2,\ldots,$$
where
$$\gamma_n=\frac{1}{n!}\int_0^{\beta R^2} x^{n}
e^{-x}\,dx=1-e^{-\beta R^2}\sum_{k=0}^n
\frac{(\beta R^2)^k}{k!}.$$
It follows that
\begin{equation}\label{eq4.18}
\|\mathbb T_f^{\beta}\|_{F_{\beta}^2}=
\sup_{n\geq 0}|\gamma_n|=1-e^{-\beta R^2}.
\end{equation}

An orthonormal basis of $A_{\alpha}^2$ is given by
$$e_n^{\alpha}(z)=\sqrt{\frac{\Gamma(n+\alpha+2)}
{n!\Gamma(\alpha+2)}}z^n,\quad n\geq 0.$$
For $r>R$, on the weighted Bergman space $A_{\alpha}^2$,
we have{\small
\begin{align*}
&\langle \mathbf T^{\alpha}_{f_{r,\sigma}}e^{\alpha}_n,
\,e^{\alpha}_m\rangle_{\alpha}\\[5pt]
&=(\alpha+1)\sqrt{\frac{\Gamma(n+\alpha+2)
\Gamma(m+\alpha+2)}{n!m!\Gamma(\alpha+2)^2}}
\ind(1-|z|^2)^{\alpha+\sigma}\mathbbm{1}_{\Omega}(rz)
z^n\bar{z}^m\,\frac{dA(z)}{\pi}\\[8pt]
&=(\alpha+1)\sqrt{\frac{\Gamma(n+\alpha+2)
\Gamma(m+\alpha+2)}{n!m!\Gamma(\alpha+2)^2}}
\int_0^{R/r}\int_0^{2\pi}(1-x^2)^{\alpha+\sigma}
x^{n+m+1}e^{i(n-m)\theta}\,\frac{d\theta dx}{\pi}\\[8pt]
&=\begin{cases}
0, & n\neq m,\\[5pt]
\displaystyle\frac{(\alpha+1)\Gamma(n+\alpha+2)}
{n!\Gamma(\alpha+2)}\int_0^{R^2/r^2}
(1-x)^{\alpha+\sigma}x^{n}dx,  &n=m.
\end{cases}
\end{align*}
}It follows that for $r>R$ we have
$$\|\mathbf T^{\alpha}_{f_{r,\sigma}}\|_{A_{\alpha}^2}=
\frac{\alpha+1}{\alpha+\sigma+1}\left[1-\left(1-
\frac{R^2}{r^2}\right)^{\alpha+\sigma+1}\right].$$
This also follows from \cite[Theorem 3.1]{RT}; see
\eqref{eq4.20}. As a consequence, we have
\begin{align*}
\lim_{r\to \infty}\|\mathbf T^{\alpha}_{f_{r,\sigma}}
\|_{A_{\beta r^2}^2}
&=\lim_{r\to \infty}\frac{\beta r^2+1}{\beta r^2+\sigma+1}
\left[1-\left(1-\frac{R^2}{r^2}\right)^{\beta r^2+\sigma+1}
\right]\\
&=1-e^{-\beta R^2}=\|\mathbb T_f\|_{F^2_{\beta}},
\end{align*}
completing the proof of the theorem.
\end{proof}

As an application of \eqref{eq4.17}, we will obtain an estimate for
the operator norm of some Toeplitz operators on Fock spaces
(see Corollary \ref{4.7}). We need the next theorem, which
appeared implicitly in \cite[Theorem 5.2]{NT}. For completeness,
we include a detailed proof here.

\begin{thm}\label{4.6}
For $f\in L^1(\D,d\lambda)\cap L^{\infty}(\D)$ we have
\begin{equation}\label{eq4.19}
\|\mathbf T_{f}^{\alpha}\|_{A_{\alpha}^2}\leq
\left(1-\frac{\|f\|_{\infty}^{\alpha+1}}{(\|f\|_{\infty}+
\|f\|_{L^1(\D,d\lambda)})^{\alpha+1}}\right)\|f\|_{\infty}.
\end{equation}
Moreover, equality holds if $f$ is the characteristic function of
a Bergman metric disc in $\D$.
\end{thm}

\begin{proof}
Let $\Omega$ be a measurable subset of $\D$ with
$\lambda(\Omega)<\infty$. By Theorem~3.1 of Ramos-Tilli
\cite{RT}, for every $g\in A^2_{\alpha}$ with $\alpha>-1$,
we have
\begin{equation}\label{eq4.20}
\int_{\Omega}|g(z)|^2\,dA_{\alpha}(z)\leq  \left[1-\frac{1}{(1+\lambda(\Omega))^{\alpha+1}}\right]\|g\|_{A^2_{\alpha}}^2.
\end{equation}
Moreover, equality holds if and only if $\Omega$ is a Bergman
metric disc $D(z_0,r)$ (up to a Bergman volume zero set) for
some $z_0\in\D,\,\, r>0$ and $g(z)=ck^{\alpha}_{z_0}(z)$.

If $f=\mathbbm{1}_{\Omega}$ for some Bergman metric disc
$\Omega$ in $\D$, equality in \eqref{eq4.19} follows from
the Ramos-Tilli result above. It remains to prove the inequality
in \eqref{eq4.19}.

By the Cauchy-Schwarz inequality, for any $g,h\in A_{\alpha}^2$,
we have
\begin{align*}
|\langle \mathbf T_f^{\alpha}g,\,h\rangle_{A_{\alpha}^2}|
&\leq \int_{\D}|f(z)||g(z)||h(z)|\,dA_{\alpha}(z)\\
&\leq \left(\int_{\D}|g(z)|^2|f(z)|\,dA_{\alpha}(z)\right)^{1/2}\left(\int_{\D}|h(z)|^2|f(z)|\,dA_{\alpha}(z)\right)^{1/2}\\
&=\langle \mathbf T_{|f|}^{\alpha}g,\,g
\rangle_{A_{\alpha}^2}^{1/2}\langle \mathbf T_{|f|}^{\alpha}h,
\,h\rangle_{A_{\alpha}^2}^{1/2}.
\end{align*}
This implies that $\|\mathbf T_{f}^{\alpha}\|_{A_{\alpha}^2}\leq
\|\mathbf T_{|f|}^{\alpha}\|_{A_{\alpha}^2}$. Thus we
may assume $f\geq 0$. Then we have
$$\|\mathbf T^{\alpha}_{f}\|_{A_{\alpha}^2}=
\sup\{\langle \mathbf T^{\alpha}_{f}g,\,g\rangle_{A_{\alpha}^2}:
g\in A_{\alpha}^2,\,\|g\|_{A_{\alpha}^2}\leq1 \}.$$
By Tonelli's theorem, we have
\begin{align}
\langle \mathbf T^{\alpha}_{f}g,\,g\rangle_{A_{\alpha}^2}
&=\ind f(z)|g(z)|^2\,dA_{\alpha}(z)\nonumber\\
&=\ind \int_0^{f(z)}|g(z)|^2\,dt\,dA_{\alpha}(z)\nonumber\\
&=\int_0^{\|f\|_{\infty}}\int_{\{z\in\D:f(z)>t\}}
|g(z)|^2\,dA_{\alpha}(z)\,dt\label{eq4.22}
\end{align}
for $g\in A^2_{\alpha}$. By \eqref{eq4.20},
\begin{equation}\label{eq4.23}
\int_{\{z\in\D:f(z)>t\}}|g(z)|^2dA_{\alpha}(z)\leq\left[1-
\frac{1}{(1+\nu(t))^{\alpha+1}}\right]\|g\|_{A_{\alpha}^2}^2,
\end{equation}
where $\nu(t)$ is the distribution function of $f$ with respect
to the M\"obius invariant area measure, namely,
$$\nu(t)=\lambda(\{z\in\D:f(z)>t\}).$$
Let $h(x)=1-\frac{1}{(1+x)^{\alpha+1}}$ for $x>0$. Then $h$
is concave. Plugging \eqref{eq4.23} into \eqref{eq4.22} and then
applying Jensen's inequality, we obtain
\begin{align*}
\langle \mathbf T_{f}g,\,g\rangle_{A_{\alpha}^2}
&\leq \left(\int_0^{\|f\|_{\infty}}\left[1-\frac{1}
{(1+\nu(t))^{\alpha+1}}\right]\,dt\right)\|g\|_{A_{\alpha}^2}^2\\
&=\|f\|_{\infty}\left(\int_0^{\|f\|_{\infty}}h(\nu(t))\,
\frac{dt}{\|f\|_{\infty}}\right)\|g\|_{A_{\alpha}^2}^2\\
&\leq\|f\|_{\infty}h\left(\int_0^{\|f\|_{\infty}}\nu(t)\,
\frac{dt}{\|f\|_{\infty}}\right)\|g\|_{A_{\alpha}^2}^2\\
&=\|f\|_{\infty}h\left(\frac{\|f\|_{L^1(\D,d\lambda)}}
{\|f\|_{\infty}}\right)\|g\|_{A_{\alpha}^2}^2\\
&=\|f\|_{\infty}\left(1-\frac{\|f\|_{\infty}^{\alpha+1}}
{(\|f\|_{\infty}+\|f\|_{L^1(\D,d\lambda)})^{\alpha+1}}
\right)\|g\|_{A_{\alpha}^2}^2.
\end{align*}
This completes the proof of the theorem.
\end{proof}

Using Theorems \ref{4.5} and \ref{4.6}, we now obtain the
following sharp norm estimate for Toeplitz operators on the
Fock space $F^2_\beta$. We will write $L^1(\C)$ for $L^1(\C,dA)$.

\begin{cor}\label{4.7}
For $f\in L^1(\C)\cap L^{\infty}(\C)$ and $\beta>0$ we have
$$\|\mathbb T^{\beta}_{f}\|_{F_{\beta}^2}\leq \|f\|_{\infty}
\left[1-\exp\left(-\frac{\beta}{\pi}\frac{\|f\|_{L^1(\C)}}
{\|f\|_{\infty}}\right)\right].$$
Equality holds if $f$ is the characteristic function of a disc in $\C$.
\end{cor}

\begin{proof}
If $f=\mathbbm{1}_{\Omega}$ for some disc $\Omega$ in
$\C$, then equality follows from \eqref{eq4.18}. It remains
to prove the inequality.

Let $f_{r,2}(z)=(1-|z|^2)^{2}f(rz)$ for $z\in\D$. Then
$f_{r,2}\in L^1(\D,d\lambda)\cap L^{\infty}(\D)$ and thus
from Theorem \ref{4.6} we deduce that
\begin{align}
\|\mathbf T^{\beta r^2}_{f_{r,2}}\|_{A_{\beta r^2}^2}
&\leq \left(1-\frac{\|f_{r,2}\|_{\infty}^{\beta r^2+1}}
{(\|f_{r,2}\|_{\infty}+\|f_{r,2}\|_{L^1(\D,d\lambda)}
)^{\beta r^2+1}}\right)\|f_{r,2}\|_{\infty}\nonumber\\[8pt]
&=\left[1-\left(1-\frac{\|f_{r,2}\|_{L^1(\D,d\lambda)}}
{\|f_{r,2}\|_{\infty}+\|f_{r,2}\|_{L^1(\D,d\lambda)}}
\right)^{\beta r^2+1}\right]\|f_{r,2}\|_{\infty}\nonumber\\[8pt]
&=\left[1-\left(1-\frac{1}{1+\frac{\|f_{r,2}\|_{\infty}}
{\|f_{r,2}\|_{L^1(\D,d\lambda)}}}\right)^{\beta r^2+1}
\right]\|f_{r,2}\|_{\infty}\label{eq4.24}
\end{align}
where $\|f_{r,2}\|_{\infty}=\|f_{r,2}\|_{L^{\infty}(\D)}$.
It is easy to see that
\begin{equation}\label{eq4.25}
\lim_{r\to \infty}\|f_{r,2}\|_{\infty}=\|f\|_{\infty}.
\end{equation}
By a change of variables, we have
$$\|f_{r,2}\|_{L^1(\D,d\lambda)}=\frac{1}{\pi}\ind|f(rz)|\,dA(z)=
\frac{1}{\pi r^2}\int_{r\D}|f(z)|\,dA(z).$$
It follows that\small{
\begin{equation}\label{eq4.26}
\lim_{r\to \infty}\left[1-\left(1-\frac{1}{1+\frac{\|f_{r,2}
\|_{\infty}}{\|f_{r,2}\|_{L^1(\D,d\lambda)}}}
\right)^{\beta r^2+1}\right]=1-\exp\left(-\frac{\beta}{\pi}
\frac{\|f\|_{L^1(\C)}}{\|f\|_{\infty}}\right).
\end{equation}}
From Theorem \ref{4.5}, together with \eqref{eq4.24},
\eqref{eq4.25} and \eqref{eq4.26}, we derive that
$$\|\mathbb T^{\beta}_f\|_{F_{\beta}^2}\leq
\limsup_{r\to\infty}\|\mathbf T^{\beta r^2}_{f_{r,2}}
\|_{A_{\beta r^2}^2}\leq
\|f\|_{\infty}\left[1-\exp\left(-\frac{\beta}{\pi}
\frac{\|f\|_{L^1(\C)}}{\|f\|_{\infty}}\right)\right].$$
This completes the proof.
\end{proof}

Finally in this section we note that, in terms of the canonical
monomial orthonormal bases $\{e_n^{\beta r^2}\}$ for
$A^2_{\beta r^2}$ and $\{\omega^\beta_n\}$ for
$F^2_\beta$, we have
$$\lim_{r\to \infty}\langle \mathbf L_{f_{r,\sigma}}^{\phi_r,\psi_r,
\beta r^2}e_n^{\beta r^2},\,e_m^{\beta r^2}
\rangle_{A_{\beta r^2}^2}=
\langle \mathbb L_f^{\phi,\psi}\omega^{\beta}_n,\,
\omega^{\beta}_m\rangle_{F_{\beta}^2}$$
for all $n\ge0$ and $m\ge0$. In fact, if we write $p_n(z)=z^n$,
then by \eqref{eq4.8} and Stirling's formula,
\begin{align*}
&\langle \mathbf L_{f_{r,\sigma}}^{\phi_r,\psi_r,\beta r^2}
e_n^{\beta r^2},\,e_m^{\beta r^2}
\rangle_{A_{\beta r^2}^2}\\[5pt]
&\quad=\sqrt{\frac{\Gamma(n+\beta r^2+1)
\Gamma(m+\beta r^2+1)}{n!m!\Gamma(\beta r^2+1)^2}}\,
\langle\mathbf L_{f_{r,\sigma}}^{\phi_r,\psi_r,\beta r^2} p_n,
\,p_m\rangle_{A_{\beta r^2}^2}\\[8pt]
&\quad=\frac{1}{r^nr^m}\sqrt{\frac{\Gamma(n+\beta r^2+1)
\Gamma(m+\beta r^2+1)}{n!m!\Gamma(\beta r^2+1)^2}}
\,\langle \mathbf L_{f_{r,\sigma}}^{\phi_r,\psi_r,\beta r^2}
(p_n)_r,\,(p_m)_r\rangle_{A_{\beta r^2}^2}\\[8pt]
&\quad\stackrel{r\to \infty}{\xrightarrow{\hspace*{1cm}}}
\sqrt{\frac{\beta ^n\beta^m}{n!m!}}\,\langle
\mathbb L_f^{\phi,\psi,\beta}p_n,\,p_m
\rangle_{F_{\beta}^2}\\[8pt]
&\quad=\langle \mathbb L_f^{\phi,\psi,\beta}\omega^{\beta}_n,
\,\omega^{\beta}_m\rangle_{F_{\beta}^2}.
\end{align*}
This limit formula is a different version of Theorem \ref{4.3}
for functions in monomial orthonormal bases.

\section{Windowed Berezin transforms and applications}

Given a window function $\psi$ in $A^2_{\alpha}$, we define
the windowed Berezin transform of a function $f$ on $\T\times\D$
associated to $\psi$ by
\begin{align*}
B_{\alpha}^{\psi}f(e^{i\theta},z)=(\alpha+1)\int_{0}^{2\pi}
\frac{dt}{2\pi}\ind f(e^{it},w)|\langle U_{z}^{\alpha}
\psi_{\theta},\,U_{w}^{\alpha}\psi_t\rangle_{A_{\alpha}^2}|^2
\,d\lambda(w),
\end{align*}
where $\psi_t(\zeta)=\psi(e^{it}\zeta)$. If $\psi=1$ and if
$f(e^{i\theta},z)=f(z)$ is independent of $\theta$, this is just the
classical $\alpha$-Berezin transform of $f$; see \cite[Section 6.3]{Zhu}.

Recall that, for $\alpha>-1$, $V^{\alpha}$ is the
unitary operator from $A^2\,(=A_0^2)$ onto $A_{\alpha}^2$ such
that $V^{\alpha}e_n^0=e_n^{\alpha}$, where $\{e_n^\alpha\}$ is
the canonical monomial orthonormal basis for $A^2_\alpha$. For
$\psi\in A^2$, we write $\psi^{\alpha}=V^{\alpha}\psi$.

We will use $L^p(\T\times \D)$ to denote $L^p(\T\times \D,dH)$,
where $dH=\frac{d\theta}{2\pi}d\lambda$ is the Haar measure on $\aut=\T\times\D$. For a
function $f$ on $\D$, we can regard $f$ as a function on
$\T\times \D$ which is independent of the first variable, namely,
$f(e^{i\theta},z)=f(z)$. In view of this, $L^p(\D,d\lambda)$
can be considered a subspace of $L^p(\T\times\D)$.

In this section we prove a limit theorem for windowned Berezin
transforms and apply it to obtain a Szeg\"o type theorem for
localization operators on weighted Bergman spaces. We begin
with several technical lemmas.

\begin{lem}\label{5.1}
Suppose $\alpha>-1$, $0<r<1$, $\psi\in A^2$ with
$\|\psi\|_{A^2}=1$, and $\psi^{\alpha}=V^{\alpha}\psi$. Then
$$\lim_{\alpha\to\infty}\int_{0}^{2\pi}\frac{dt}{2\pi}
\int_{r<|w|<1}|\langle U_{w}^{\alpha}(\psi^{\alpha})_t,
\,\psi^{\alpha}\rangle_{A_{\alpha}^2}|^2\,d\lambda(w)=0,$$
where $(\psi^\alpha)_t(z)=\psi^\alpha(e^{it}z)$.
\end{lem}

\begin{proof}
For any $0<\varepsilon<1$ there is a polynomial $p$ such that
$\|p-\psi\|_{A^2}<\varepsilon$. If we write $\psi^{\alpha}=
p^{\alpha}+(\psi^{\alpha}-p^{\alpha})$, then
\begin{align*}
&\int_{0}^{2\pi}\frac{dt}{2\pi}\int_{r<|w|<1}|\langle
U_{w}^{\alpha}(\psi^{\alpha})_t,\,\psi^{\alpha}
\rangle_{A_{\alpha}^2}|^2\,d\lambda(w)\\
&=\int_{0}^{2\pi}\frac{dt}{2\pi}\int_{r<|w|<1}|\langle
U_{w}^{\alpha}[p^{\alpha}+(\psi^{\alpha}-
p^{\alpha})]_t,\,[p^{\alpha}+(\psi^{\alpha}-p^{\alpha})]
\rangle_{A_{\alpha}^2}|^2\,d\lambda(w)\\
&\leq 4\int_{0}^{2\pi}\frac{dt}{2\pi}\int_{r<|w|<1}|\langle
U_{w}^{\alpha}(p^{\alpha})_t,\,p^{\alpha}
\rangle_{A_{\alpha}^2}|^2\,d\lambda(w)\\
&\quad+4\int_{0}^{2\pi}\frac{dt}{2\pi}\int_{r<|w|<1}|
\langle U_{w}^{\alpha}(p^{\alpha})_t,\,\psi^{\alpha}
-p^{\alpha}\rangle_{A_{\alpha}^2}|^2\,d\lambda(w)\\
&\quad+4\int_{0}^{2\pi}\frac{dt}{2\pi}\int_{r<|w|<1}|
\langle U_{w}^{\alpha}(\psi^{\alpha}-p^{\alpha})_t,
\,p^{\alpha}\rangle_{A_{\alpha}^2}|^2\,d\lambda(w)\\
&\quad +4\int_{0}^{2\pi}\frac{dt}{2\pi}\int_{r<|w|<1}|
\langle U_{w}^{\alpha}(\psi^{\alpha}-p^{\alpha})_t,
\,\psi^{\alpha}-p^{\alpha}\rangle_{A_{\alpha}^2}|^2
\,d\lambda(w).
\end{align*}
By Theorem \ref{3.4}, we have
\begin{align*}
4(\alpha+1)&\int_{0}^{2\pi}\frac{dt}{2\pi}\int_{r<|w|<1}|
\langle U_{w}^{\alpha}(p^{\alpha})_t,\,\psi^{\alpha}-
p^{\alpha}\rangle_{A_{\alpha}^2}|^2
\,d\lambda(w)\\[8pt]
&\leq 4\|p^{\alpha}\|_{A^2_{\alpha}}^2\|\psi^{\alpha}-
p^{\alpha}\|_{A^2_{\alpha}}^2
=4\|p\|_{A^2}^2\|\psi-p\|_{A^2}^2<16\varepsilon^2.
\end{align*}
Similarly, we have
$$4(\alpha+1)\int_{0}^{2\pi}\frac{dt}{2\pi}
\int_{r<|w|<1}|\langle U_{w}^{\alpha}
(\psi^{\alpha}-p^{\alpha})_t,\,p^{\alpha}
\rangle_{A_{\alpha}^2}|^2\,d\lambda(w)<16\varepsilon^2,$$
and
$$4(\alpha+1)\int_{0}^{2\pi}\frac{dt}{2\pi}\int_{r<|w|<1}
|\langle U_{w}^{\alpha}(\psi^{\alpha}-p^{\alpha})_t,
\,\psi^{\alpha}-p^{\alpha}\rangle_{A_{\alpha}^2}|^2
\,d\lambda(w)<4\varepsilon^4.$$
To finish the proof, we only need to show that
\begin{equation}\label{eq5.1}
\lim_{\alpha\to \infty}\int_{0}^{2\pi}\frac{dt}{2\pi}
\int_{r<|w|<1}|\langle U_{w}^{\alpha}(p^{\alpha})_t,
\,p^{\alpha}\rangle_{A_{\alpha}^2}|^2\,d\lambda(w)=0.
\end{equation}
Write
$$p=\sum_{j=0}^na_je^0_j.$$
Then
$$p^{\alpha}:=V^{\alpha}p=\sum_{j=1}^na_je_j^{\alpha}.$$
By the Cauchy-Schwarz inequality, we have
\begin{align}\label{eq5.2}
&\int_{0}^{2\pi}\frac{dt}{2\pi}\int_{r<|w|<1}|\langle
U_{w}^{\alpha}(p^{\alpha})_t,\,p^{\alpha}
\rangle_{A_{\alpha}^2}|^2\,d\lambda(w)\nonumber\\
&=\int_{0}^{2\pi}\frac{dt}{2\pi}\int_{r<|w|<1}\bigg|
\sum_{j=0}^n\sum_{k=0}^na_j\overline{a_k}
\langle U_{w}^{\alpha}(e_j^{\alpha})_t,\,e_k^{\alpha}
\rangle_{A_{\alpha}^2}\bigg|^2\,d\lambda(w)\nonumber\\
&\leq\left(\sum_{j=0}^n|a_j|^2\right)^2\sum_{j=0}^n
\sum_{k=0}^n\int_{0}^{2\pi}\frac{dt}{2\pi}\int_{r<|w|<1}
\big|\langle U_{w}^{\alpha}(e_j^{\alpha})_t,\,e_k^{\alpha}
\rangle_{A_{\alpha}^2}\big|^2\,d\lambda(w).
\end{align}
A straightforward calculation shows that{\small
\begin{align*}
&U_{w}^{\alpha}(e_j^{\alpha})_t(\zeta)
=\sqrt{\frac{\Gamma(j+\alpha+2)}{j!\Gamma(\alpha+2)}}
\big[e^{it}\varphi_w(\zeta)\big]^{j}\big[
\varphi_w'(\zeta)\big]^{1+\frac{\alpha}{2}}\\[8pt]
&=e^{ijt}\sqrt{\frac{\Gamma(j+\alpha+2)}
{j!\Gamma(\alpha+2)}}\left(\frac{\zeta-w}
{1-\bar w\zeta}\right)^{j}\left[\frac{1-|w|^2}
{(1-\bar w\zeta)^2}\right]^{1+\frac{\alpha}{2}}\\[8pt]
&=e^{ijt}(1-|w|^2)^{1+\frac{\alpha}{2}}
\sqrt{\frac{\Gamma(j+\alpha+2)}{j!\Gamma(\alpha+2)}}
\frac{(\zeta-w)^j}{(1-\bar w\zeta)^{j+2+\alpha}}\\[8pt]
&=e^{ijt}(1-|w|^2)^{1+\frac{\alpha}{2}}
\sqrt{\frac{\Gamma(j+\alpha+2)}{j!\Gamma(\alpha+2)}}
\sum_{l=0}^j\sum_{m=0}^{\infty}{j \choose l}
\frac{\Gamma(m+j+2+\alpha)}{m!\Gamma(j+2+\alpha)}
(-w)^{j-l}{\bar w}^m\zeta^{m+l}.
\end{align*}
}It follows from Stirling's formula that there is a constant $C_{j,k}$
depending on $j$ and $k$ such that{\small
\begin{align}\label{eq5.3}
&\big|\langle U_{w}^{\alpha}(e_j^{\alpha})_t,
\,e_k^{\alpha}\rangle_{A_{\alpha}^2}\big|^2\nonumber\\[8pt]
&=(1-|w|^2)^{2+\alpha}\frac{k!\Gamma(j+\alpha+2)}
{j!\Gamma(k+\alpha+2)}\left|\sum_{\left\{(l,m)\in
\mathbb{N}^2 :l+m=k;l\leq j\right\}}{j \choose l}
\frac{\Gamma(m+j+2+\alpha)}{m!\Gamma(j+2+\alpha)}
(-w)^{j-l}{\bar w}^m\right|^2\nonumber\\[8pt]
&\leq C_{j,k}(\alpha+1)^{j+k}(1-|w|^2)^{2+\alpha}.
\end{align}
}Combining \eqref{eq5.2} and \eqref{eq5.3}, we obtain
\begin{align}
&(\alpha+1)\int_{0}^{2\pi}\frac{dt}{2\pi}\int_{r<|w|<1}
|\langle U_{w}^{\alpha}(p^{\alpha})_t,\,p^{\alpha}
\rangle_{A_{\alpha}^2}|^2\,d\lambda(w)\nonumber\\
&\leq  \left(\sum_{j=0}^n|a_j|^2\right)^2
\left[\sum_{j=0}^n\sum_{k=0}^nC_{j,k}
(\alpha+1)^{j+k+1}\right]\int_{r<|w|<1}
(1-|w|^2)^{\alpha}\,\frac{dA(w)}{\pi}\nonumber\\
&=\left(\sum_{j=0}^n|a_j|^2\right)^2\left[\sum_{j=0}^n
\sum_{k=0}^nC_{j,k}(\alpha+1)^{j+k}\right]
(1-r^2)^{\alpha+1}.\label{eq5.4}
\end{align}
Equation \eqref{eq5.1} now follows from \eqref{eq5.4} by
letting $\alpha$ goes to infinity. This finishes the proof.
\end{proof}

\begin{lem}\label{5.2}
Suppose $\psi\in A^2$ with $\|\psi\|_{A^2}=1$, $\psi^{\alpha}=V^{\alpha}\psi$, and
$f\in C(\D)\cap L^{\infty}(\D)$. Then
$$\lim_{\alpha\to\infty}B_\alpha^{\psi^\alpha}f(e^{i\theta},z)=f(z)$$
for all $(e^{i\theta},z)\in\T\times \D$.
\end{lem}

\begin{proof}
Recall that we regard $f$ as a function on $\T\times \D$ which is
independent of the first variable. By a change of variables,
\eqref{eq3.3}, and \eqref{eq3.4}, we have\small{
\begin{align*}
B_{\alpha}^{\psi^{\alpha}}f(e^{i\theta},z)
&=(\alpha+1)\int_{0}^{2\pi}\frac{dt}{2\pi}\ind f(e^{it},w)|
\langle U_{z}^{\alpha}(\psi^{\alpha})_{\theta},
\,U_{w}^{\alpha}(\psi^{\alpha})_{t}
\rangle_{A_{\alpha}^2}|^2\,d\lambda(w)\\
&=(\alpha+1)\int_{0}^{2\pi}\frac{dt}{2\pi}
\ind f(e^{it},w)|\langle U_{e^{i\theta},z}^{\alpha}
\psi^{\alpha},\,U_{e^{it},w}^{\alpha}\psi^{\alpha}
\rangle_{A_{\alpha}^2}|^2\,d\lambda(w)\\
&=(\alpha+1)\int_{0}^{2\pi}\frac{dt}{2\pi}
\ind f(e^{it},w)|\langle (U_{e^{it},w}^{\alpha})^{\ast}
U_{e^{i\theta},z}^{\alpha}\psi^{\alpha},\,
\psi^{\alpha}\rangle_{A_{\alpha}^2}|^2\,d\lambda(w)\\
&=(\alpha+1)\int_{0}^{2\pi}\frac{dt}{2\pi}\ind f(e^{it},w)
|\langle U_{(e^{it},w)^{-1}}^{\alpha}
U_{e^{i\theta},z}^{\alpha}\psi^{\alpha},\,\psi^{\alpha}
\rangle_{A_{\alpha}^2}|^2\,d\lambda(w)\\
&=(\alpha+1)\int_{0}^{2\pi}\frac{dt}{2\pi}\ind f(e^{it},w)
|\langle U_{(e^{i\theta},z)\cdot(e^{it},w)^{-1}}^{\alpha}
\psi^{\alpha},\,\psi^{\alpha}\rangle_{A_{\alpha}^2}|^2
\,d\lambda(w)\\
&=(\alpha+1)\int_{0}^{2\pi}\frac{dt}{2\pi}
\ind f((e^{it},w)^{-1}\cdot(e^{i\theta},z))|
\langle U_{e^{it},w}^{\alpha}\psi^{\alpha},
\,\psi^{\alpha}\rangle_{A_{\alpha}^2}|^2\,d\lambda(w)\\
&=(\alpha+1)\int_{0}^{2\pi}\frac{dt}{2\pi}\ind
f(\varphi_{-z}(-e^{i(t-\theta)}w))|\langle
U_{e^{it},w}^{\alpha}\psi^{\alpha},\,\psi^{\alpha}
\rangle_{A_{\alpha}^2}|^2\,d\lambda(w)\\
&=(\alpha+1)\int_{0}^{2\pi}\frac{dt}{2\pi}
\ind f(\varphi_{-z}(-w))|\langle U_{e^{it},
e^{i(\theta-t)}w}^{\alpha}\psi^{\alpha},\,\psi^{\alpha}
\rangle_{A_{\alpha}^2}|^2\,d\lambda(w).
\end{align*}
}Since $f$ is continuous at $z$, for any $\varepsilon>0$ there is
a $\delta>0$ such that $|f(w)-f(z)|<\varepsilon$ whenever
$|w-z|<\delta$. Since $\|\psi^{\alpha}\|_{A_{\alpha}^2}=1$, by Theorem \ref{3.4} for any fixed $r\in(0, \delta/2)$ we have\small{
\begin{align}\label{eq5.5}
&B_{\alpha}^{\psi^{\alpha}}f(e^{i\theta},z)-f(z)\\
&=(\alpha+1)\int_{0}^{2\pi}\frac{dt}{2\pi}
\ind[f(\varphi_{-z}(-w))-f(z)]|\langle U_{e^{it},
e^{i(\theta -t)}w}^{\alpha}\psi^{\alpha},\,\psi^{\alpha}
\rangle_{A_{\alpha}^2}|^2\,d\lambda(w)\nonumber\\
&=(\alpha+1)\int_{0}^{2\pi}\frac{dt}{2\pi}\int_{\{|w|<r\}}
[f(\varphi_{-z}(-w))-f(z)]|\langle U_{e^{it},
e^{i(\theta-t)}w}^{\alpha}\psi^{\alpha},\,\psi^{\alpha}
\rangle_{A_{\alpha}^2}|^2\,d\lambda(w)\nonumber\\
&+(\alpha+1)\int_{0}^{2\pi}\frac{dt}{2\pi}\int_{\{r<|w|<1\}}
[f(\varphi_{-z}(-w))-f(z)]|\langle U_{e^{it},
e^{i(\theta-t)}w}^{\alpha}\psi^{\alpha},\,\psi^{\alpha}
\rangle_{A_{\alpha}^2}|^2\,d\lambda(w).\nonumber
\end{align}
}For $|w|<r$, we have
$$|\varphi_{-z}(-w)-z|=\left|\frac{-w+z}{1-\bar zw}-z\right|=
\frac{(1-|z|^2)|w|}{|1-\bar zw|}\leq(1+|z|)|w|<2|w|<\delta.$$
We deduce from Theorem \ref{3.4} and the continuity
of $f$ that\small{
\begin{align}
&(\alpha+1)\int_{0}^{2\pi}\frac{dt}{2\pi}\int_{\{|w|<r\}}
|f(\varphi_{-z}(-w))-f(z)||\langle U_{e^{it},
e^{i(\theta-t)}w}^{\alpha}\psi^{\alpha},
\,\psi^{\alpha}\rangle_{A_{\alpha}^2}|^2
\,d\lambda(w)\nonumber\\
&<\varepsilon(\alpha+1)\int_{0}^{2\pi}\frac{dt}{2\pi}
\int_{\{|w|<r\}}|\langle U_{e^{it},
e^{i(\theta-t)}w}^{\alpha}\psi^{\alpha},\,\psi^{\alpha}
\rangle_{A_{\alpha}^2}|^2\,d\lambda(w)\nonumber\\
&\leq \varepsilon.\label{eq5.6}
\end{align}
}By Lemma \ref{5.1} and the boundedness of $f$, we have\small{
\begin{align}
&(\alpha+1)\int_{0}^{2\pi}\frac{dt}{2\pi}\int_{\{r<|w|<1\}}
|f(\varphi_{-z}(-w))-f(z)||\langle U_{e^{it},
e^{i(\theta-t)}w}^{\alpha}\psi^{\alpha},\,\psi^{\alpha}
\rangle_{A_{\alpha}^2}|^2\,d\lambda(w)\nonumber\\
&\leq 2\|f\|_{\infty}(\alpha+1)\int_{0}^{2\pi}\frac{dt}{2\pi}
\int_{\{r<|w|<1\}}|\langle U_{e^{it},
e^{i(\theta-t)}w}^{\alpha}\psi^{\alpha},\,\psi^{\alpha}
\rangle_{A_{\alpha}^2}|^2\,d\lambda(w)\nonumber\\
&= 2\|f\|_{\infty}(\alpha+1)\int_{0}^{2\pi}\frac{dt}{2\pi}
\int_{\{r<|w|<1\}}|\langle U_{e^{it},w}^{\alpha}\psi^{\alpha},\,\psi^{\alpha}
\rangle_{A_{\alpha}^2}|^2\,d\lambda(w)\nonumber\\
&= 2\|f\|_{\infty}(\alpha+1)\int_{0}^{2\pi}\frac{dt}{2\pi}
\int_{\{r<|w|<1\}}|\langle U_{w}^{\alpha}(\psi^{\alpha})_t,\,\psi^{\alpha}
\rangle_{A_{\alpha}^2}|^2\,d\lambda(w)\nonumber\\
&\rightarrow 0\quad{\rm as}\quad \alpha\to\infty.\label{eq5.7}
\end{align}}
The desired result now follows from \eqref{eq5.5}, \eqref{eq5.6},
and \eqref{eq5.7}.
\end{proof}

\begin{lem}\label{5.3}
Suppose $\psi\in A^2_{\alpha}$ with $\|\psi\|_{A_{\alpha}^2}=1$
and $1\leq p\leq \infty$. Then $B_{\alpha}^{\psi}$ is
bounded from $L^p(\T\times \D)$ to $L^p(\T\times \D)$. Moreover,
the operator norm of $B_{\alpha}^{\psi}$ satisfies
\begin{enumerate}
\item [(a)] $\|B_{\alpha}^{\psi}\|_{L^p(\T\times \D)\to
L^p(\T\times \D)}=1$ if $ p=1$ or $p=\infty$.
\item [(b)] $\|B_{\alpha}^{\psi}\|_{L^p(\T\times \D)\to
L^p(\T\times \D)}\leq1$ if $ 1<p<\infty$.
\end{enumerate}
\end{lem}

\begin{proof}
For $p=1$, it follows from Theorem \ref{3.4} and Fubini's
theorem that{\small
\begin{align}
&\|B_{\alpha}^{\psi}f\|_{L^1(\T\times\D)}=\int_{0}^{2\pi}
\frac{d\theta}{2\pi}\ind|B_{\alpha}^{\psi}f(e^{i\theta},z)|
\,d\lambda(z)\nonumber\\
&\leq\int_{0}^{2\pi}\frac{d\theta}{2\pi}\ind\left((\alpha+1)
\int_{0}^{2\pi}\frac{dt}{2\pi}
\ind|f(e^{it},w)||\langle U_{z}^{\alpha}\psi_{\theta},
\,U_{w}^{\alpha}\psi_{t}\rangle_{A_{\alpha}^2}|^2
\,d\lambda(w)\right)\,d\lambda(z)\nonumber\\
&=\int_{0}^{2\pi}\frac{dt}{2\pi}\ind|f(e^{it},w)|\left((\alpha+1)
\int_{0}^{2\pi}\frac{d\theta}{2\pi}\ind|\langle U_{z}^{\alpha}
\psi_{\theta},\,U_{w}^{\alpha}\psi_{t}\rangle_{A_{\alpha}^2}|^2
\,d\lambda(z)\right)\,d\lambda(w)\nonumber\\
&=\int_{0}^{2\pi}\frac{dt}{2\pi}\ind|f(e^{it},w)|
\,d\lambda(w)\nonumber\\
&=\|f\|_{L^1(\T\times \D)}.\label{eq5.8}
\end{align}}
If $f$ is non-negative, then $B_{\alpha}^{\psi}f(e^{i\theta},z)$
is also non-negative. As a result, in \eqref{eq5.8} equality holds
if $f$ is non-negative. Thus
$\|B_{\alpha}^{\psi}\|_{L^1(\T\times \D)\to L^1(\T\times \D)}=1$.

For $p=\infty$, we use Theorem \ref{3.4} to get
\begin{align*}
|B_{\alpha}^{\psi}f(e^{i\theta},z)|
&\leq(\alpha+1)\int_{0}^{2\pi}\frac{dt}{2\pi}\ind|f(e^{it},w)|
|\langle U_{z}^{\alpha}\psi_{\theta},\,U_{w}^{\alpha}
\psi_{t}\rangle_{A_{\alpha}^2}|^2\,d\lambda(w)\\
&\leq\|f\|_{\infty}(\alpha+1)\int_{0}^{2\pi}\frac{dt}{2\pi}
\ind|\langle U_{z}^{\alpha}\psi_{\theta},\,U_{w}^{\alpha}
\psi_{t}\rangle_{A_{\alpha}^2}|^2\,d\lambda(w)\\
&=\|f\|_{\infty}.
\end{align*}
For any positive integer $n$, let $f_n$ be the characteristic function of
$$\Omega_n=\left\{z\in\D:|z|<1-\frac{1}{n}\right\}.$$
We have{\small
$$\frac{\|B_{\alpha}^{\psi}f_n\|_{\infty}}{\|f_n\|_{\infty}}\geq
\frac{|B_{\alpha}^{\psi}f_n(e^{i\theta},z)|}{\|f_n\|_{\infty}}
=(\alpha+1)\int_{0}^{2\pi}\frac{dt}{2\pi}\int_{\Omega_n}|
\langle U_{z}^{\alpha}\psi_{\theta},\,
U_{w}^{\alpha}\psi_{t}\rangle_{A_{\alpha}^2}|^2
\,d\lambda(w)\to1$$
}as $n\to\infty$. Thus $\|B_{\alpha}^{\psi}\|_{L^{\infty}
(\T\times \D)\to L^{\infty}(\T\times \D)}=1$.

For $1<p<\infty$, the desired conclusion follows from an application of the Riesz-Thorin interpolation theorem. This completes the proof.
\end{proof}

We now prove the first main result of this section.

\begin{thm}\label{5.4}
Let $\psi\in A^2$ with $\|\psi\|_{A^2}=1$, $\psi^{\alpha}=V^{\alpha}\psi$, and $1\leq p<\infty$.
We have
$$\lim_{\alpha\to\infty}\|B_{\alpha}^{\psi^{\alpha}}f-
f\|_{L^p(\T\times \D)}=0$$
for any $f\in L^p(\D,d\lambda)$,
\end{thm}

\begin{proof}
For any $\varepsilon>0$ there is $g\in C_c(\D)$ such that
$$\|g-f\|_{L^p(\D,d\lambda)}<\varepsilon.$$
An application of Lemma \ref{5.3} and the triangle inequality give
\begin{align*}
&\|B_{\alpha}^{\psi^{\alpha}}f-f\|_{L^p(\T\times \D)}
=\|B_{\alpha}^{\psi^{\alpha}}f-B_{\alpha}^{\psi^{\alpha}}g
+B_{\alpha}^{\psi^{\alpha}}g-g+g-f\|_{L^p(\T\times \D)}\\
&\qquad\leq \|B_{\alpha}^{\psi^{\alpha}}(f-g)\|_{L^p(\T\times \D)}
+\|B_{\alpha}^{\psi^{\alpha}}g-g\|_{L^p(\T\times \D)}
+\|g-f\|_{L^p(\D,d\lambda)}\\
&\qquad\leq 2\|g-f\|_{L^p(\D,d\lambda)}+
\|B_{\alpha}^{\psi^{\alpha}}g-g\|_{L^p(\T\times \D)}\\
&\qquad<2\varepsilon+\|B_{\alpha}^{\psi^{\alpha}}g-
g\|_{L^p(\T\times \D)}.
\end{align*}
By Lemma \ref{5.2}, we have
$$\lim_{\alpha\to \infty}B_{\alpha}^{\psi^{\alpha}}
g(e^{i\theta},z)=g(z),\quad (e^{i\theta},z)\in \T\times \D.$$
It is clear that
$$|B_{\alpha}^{\psi^{\alpha}}g(e^{i\theta},z)-g(z)|^p\leq
2^{p-1}(|B_{\alpha}^{\psi^{\alpha}}
g(e^{i\theta},z)|^p+|g(z)|^p)$$
By Theorem \ref{3.4} and H\"{o}lder's inequality,
$$|B_{\alpha}^{\psi^{\alpha}}g(e^{i\theta},z)|^p\leq
(B_{\alpha}^{\psi^{\alpha}}|g|^p)(e^{i\theta},z).$$
An application of Fubini's theorem and Theorem \ref{3.4} then gives
$$\int_{0}^{2\pi}\frac{d\theta}{2\pi}
\ind(B_{\alpha}^{\psi^{\alpha}}|g|^p)(e^{i\theta},z)
\,d\lambda(z)=\|g\|_{L^p(\D,d\lambda)}^p.$$
It follows from the dominated convergence theorem that
$$\lim_{\alpha\to \infty}\|B_{\alpha}^{\psi^{\alpha}}g-
g\|_{L^p(\T\times \D)}=0,$$
which completes the proof of the theorem.
\end{proof}

As an application of Theorem~\ref{5.4}, we will obtain a
Szeg\"{o}-type theorem (Theorem~\ref{5.6}) for localization
operators on weighted Bergman spaces. To simplify notation, we
write $\mathbf L_f^{\psi,\alpha}:=\mathbf L_f^{\psi,\psi,\alpha}$.

\begin{lem}\label{5.5}
Suppose $\psi\in A^2$ with $\|\psi\|_{A^2}=1$, $\psi^{\alpha}=V^{\alpha}\psi$, and $f$ is a
non-negative function in $L^1(\D,d\lambda)\cap L^{\infty}(\D)$.
Then we have
$$\lim_{\alpha\to \infty}\frac{{\rm tr}\,(\mathbf L^{\psi^{\alpha},
\alpha}_f\mathbf L^{\psi^{\alpha},\alpha}_f-
\mathbf L^{\psi^{\alpha},\alpha}_{f^2})}{\alpha+1}=0.$$
\end{lem}

\begin{proof}
By \eqref{eq3.12}, we have
$$\mathbf L^{\psi^{\alpha},\alpha}_f=(\alpha+1)\int_0^{2\pi}
\frac{d\theta}{2\pi}\ind f(e^{i\theta},z)\big[(U^{\alpha}_{z}
(\psi^{\alpha})_\theta)\otimes
(U^{\alpha}_{z}(\psi^{\alpha})_\theta)\big]\,d\lambda(z).$$
So we can write the product
$\mathbf L^{\psi^{\alpha},\alpha}_f
\mathbf L^{\psi^{\alpha},\alpha}_f$
as the following operator-valued integral:
\begin{align*}
(\alpha+1)^2\int_{0}^{2\pi}\frac{dt}{2\pi}\ind\left(\int_{0}^{2\pi}
\frac{d\theta}{2\pi}\ind F(z,w)\big[(U^\alpha_{z}(\psi^{\alpha})_{\theta})\otimes(U^{\alpha}_{w}(\psi^{\alpha})_t)\big]
\,d\lambda(z)\right)\,d\lambda(w),
\end{align*}
where
$$F(z,w)=f(z)f(w)\langle U^{\alpha}_{w}(\psi^{\alpha})_t,\,
U^{\alpha}_{z}(\psi^{\alpha})_{\theta}\rangle_{A_{\alpha}^2}.$$
It follows that{\small
\begin{align*}
&\frac{{\rm tr}\,(\mathbf L^{\psi^{\alpha},\alpha}_f
\mathbf L^{\psi^{\alpha},\alpha}_f)}{\alpha+1}\\
&=(\alpha+1)\int_{0}^{2\pi}\frac{dt}{2\pi}\ind\left[\int_{0}^{2\pi}
\frac{d\theta}{2\pi}\ind f(z)f(w)
|\langle U^{\alpha}_{w}(\psi^{\alpha})_t,\,
U^{\alpha}_{z}(\psi^{\alpha})_{\theta}
\rangle_{A_{\alpha}^2}|^2\,d\lambda(z)\right]\,d\lambda(w)\\
&=\int_{0}^{2\pi}\frac{d\theta}{2\pi}\ind f(z)
\left[(\alpha+1)\int_{0}^{2\pi}\frac{dt}{2\pi}
\ind f(w)|\langle U^{\alpha}_{w}(\psi^{\alpha})_t,\,
U^{\alpha}_{z}(\psi^{\alpha})_\theta
\rangle_{A_{\alpha}^2}|^2\,d\lambda(w)\right]
\,d\lambda(z)\\
&=\int_{0}^{2\pi}\frac{d\theta}{2\pi}\ind f(z)
B_{\alpha}^{\psi^{\alpha}}f(e^{i\theta},z)\,d\lambda(z)
\end{align*}
}Since $f\in L^1(\D,d\lambda)\cap L^{\infty}(\D)
\subset L^2(\D,d\lambda)$, Theorem \ref{5.4} gives
$$\lim_{\alpha\to \infty}\|B_{\alpha}^{\psi^{\alpha}}f-f
\|_{L^2(\T\times\D)}=0.$$
Then we have
\begin{align}
\lim_{\alpha\to \infty}\frac{{\rm tr}
\,(\mathbf L^{\psi^{\alpha},\alpha}_f
\mathbf L^{\psi^{\alpha},\alpha}_f)}{\alpha+1}
&=\lim_{\alpha\to \infty}\int_{0}^{2\pi}\frac{d\theta}{2\pi}
\ind f(z)B_{\alpha}^{\psi^{\alpha}}f(e^{i\theta},z)
\,d\lambda(z)\nonumber\\
&=\int_{0}^{2\pi}\frac{d\theta}{2\pi}\ind f(z)f(z)
\,d\lambda(z)\nonumber\\
&=\ind f(z)f(z)\,d\lambda(z).\label{eq5.9}
\end{align}
By \eqref{eq3.12}, we have
$${\rm tr}\,(\mathbf L^{\psi^{\alpha},\alpha}_{f^2})
=(\alpha+1)\ind f(z)^2\,d\lambda(z).$$
This together with \eqref{eq5.9} implies that
$$\lim_{\alpha\to \infty}\frac{{\rm tr}\,(\mathbf L^{\psi^{\alpha},
\alpha}_f\mathbf L^{\psi^{\alpha},\alpha}_f-
\mathbf L^{\psi^{\alpha},\alpha}_{f^2})}{\alpha+1}=0,$$
completing the proof of the theorem.
\end{proof}

We now arrive at the second main result of this section.

\begin{thm}\label{5.6}
Suppose $\psi\in A^2$ with $\|\psi\|_{A^2}=1$, $\psi^{\alpha}=V^{\alpha}\psi$, and $f$ is a
non-negative function in $L^1(\D,d\lambda)\cap L^{\infty}(\D)$.
If $h\in C[0,\|f\|_{\infty}]$, then
\begin{equation}\label{eq5.10}
\lim_{\alpha\to \infty}\frac{{\rm tr}(\mathbf L^{\psi^{\alpha},\alpha}_f
h(\mathbf L^{\psi^{\alpha},\alpha}_f))}{\alpha+1}=
\ind f(z)h(f(z))\,d\lambda(z).
\end{equation}
\end{thm}

We will give two different proofs for the theorem. According
to \eqref{eq3.12}, if $\phi,\psi\in A_{\alpha}^2$ and
$f\in L^1(\T\times\D)$, then
\begin{equation}\label{eq5.11}
{\rm tr}\,(\mathbf L_{f}^{\phi,\psi,\alpha})=(\alpha+1)\langle\phi,
\,\psi\rangle_{A_{\alpha}^2}\int_0^{2\pi}\frac{d\theta}{2\pi}
\ind f(e^{i\theta},z)\,d\lambda(z).
\end{equation}
By duality, we have
\begin{equation}\label{eq5.12}
|{\rm tr}\,(\mathbf L^{\psi^{\alpha},\alpha}_fh(\mathbf L^{\psi^{\alpha},
\alpha}_f))|\leq\|h(\mathbf L^{\psi^{\alpha},\alpha}_f)\|_{A_{\alpha}^2}
\|\mathbf L^{\psi^{\alpha},\alpha}_f\|_{\mathcal S^1(A_{\alpha}^2)}
\leq\|h\|_{\infty}\|\mathbf L^{\psi^{\alpha},\alpha}_f
\|_{\mathcal S^1(A_{\alpha}^2)},
\end{equation}
and
\begin{equation}\label{eq5.13}
\|\mathbf L^{\psi^{\alpha},\alpha}_f\|_{\mathcal S^1(A_{\alpha}^2)}
=(\alpha+1)\int_0^{2\pi}\frac{d\theta}{2\pi}\ind f(e^{i\theta},z)
\,d\lambda(z),
\end{equation}
where $\|\cdot\|_{\mathcal S^1(A_{\alpha}^2)}$ is the trace
norm. Moreover, for any polynomial $p$, we have
\begin{align}
&\left|\frac{{\rm tr}\,(\mathbf L^{\psi^{\alpha},\alpha}_f
h(\mathbf L^{\psi^{\alpha},\alpha}_f))}{\alpha+1}-\ind f(z)h(f(z))
\,d\lambda(z)\right|
\leq\left|\frac{{\rm tr}\,\big[\mathbf L^{\psi^{\alpha},\alpha}_f(h-p)
(\mathbf L^{\psi^{\alpha},\alpha}_f)\big]}{\alpha+1}\right|\nonumber\\
&\qquad\qquad+\left|\frac{{\rm tr}(\mathbf L^{\psi^{\alpha},\alpha}_f
p(\mathbf L^{\psi^{\alpha},\alpha}_f))}{\alpha+1}-
\ind f(z)p(f(z))\,d\lambda(z)\right|\nonumber\\
&\qquad\qquad+\left|\ind f(z)\big[p(f(z))-h(f(z))\big]
\,d\lambda(z)\right|.\label{eq5.14}
\end{align}
From \eqref{eq5.12}, \eqref{eq5.13}, and \eqref{eq5.14},
we note that, in order to prove Theorem \ref{5.6}, it
suffices to show that \eqref{eq5.10} holds for $h(x)=x^n$
with $n=0,1,\cdots$. To do this, we first show that the
localization operator $\mathbf L^{\psi^{\alpha},\alpha}_f$ is
a Toeplitz type operator.

For $\psi\in A_{\alpha}^2$ with $\|\psi\|_{A_\alpha^2}=1$, we
define a linear operator $V^{\alpha}_{\psi}$ on $A_{\alpha}^2$ by
$$V^{\alpha}_{\psi}f(e^{i\theta},z)=\langle f,\,U^{\alpha}_{z}
\psi_{\theta}\rangle_{A_{\alpha}^2},\quad f\in A_{\alpha}^2,$$
where $\psi_{\theta}(\zeta)=\psi(e^{i\theta}\zeta)$. By
Theorem \ref{3.4}, we have
$$\langle f,\,g\rangle_{A_{\alpha}^2}=\langle V^{\alpha}_{\psi}f,
\,V^{\alpha}_{\psi}g\rangle_{L^2(\T\times \D,(\alpha+1)dH)},
\quad f,g\in A_{\alpha}^2.$$
This implies that $V^{\alpha}_{\psi}$ is an isometry from
$A_{\alpha}^2$ to $L^2(\T\times \D,(\alpha+1)dH)$. As a
result, the image of $V^{\alpha}_{\psi}$ is a closed
subspace of $L^2(\T\times \D,(\alpha+1)dH)$. We
denote by $\mathcal V^{\alpha}_{\psi}$ the image of
$V^{\alpha}_{\psi}$, namely,
$$\mathcal V^{\alpha}_{\psi}=\{V^{\alpha}_{\psi}f:f\in A_{\alpha}^2\}.$$
Then $V^{\alpha}_{\psi}$ is a unitary operator from $A_{\alpha}^2$
onto $\mathcal V^{\alpha}_{\psi}$. By Theorem \ref{3.4}, the inverse
of $V^{\alpha}_{\psi}$ is given by
$$(V^{\alpha}_{\psi})^{-1}F=(\alpha+1)\int_0^{2\pi}
\frac{d\theta}{2\pi}\ind F(z)(U^{\alpha}_{z}\psi_{\theta})
d\lambda(z),\quad F\in \mathcal V^{\alpha}_{\psi}.$$
For any $(e^{i\theta},z)\in\T\times\D$ and $f\in A_{\alpha}^2$,
from Theorem \ref{3.4} we derive that
$$|V_{\psi}^{\alpha}f(e^{i\theta},z)|=|\langle f,\,U^{\alpha}_{z}
\psi_{\theta}\rangle_{\alpha}|\leq \|f\|_{A_{\alpha}^2}
\|\psi\|_{A_{\alpha}^2}=\|V^{\alpha}_{\psi}f
\|_{L^2(\T\times \D,(\alpha+1)dH)}. $$
Thus $\mathcal V^{\alpha}_{\psi}$ is a reproducing kernel Hilbert
space. For any $F\in \mathcal V^{\alpha}_{\psi}$, we have
$$F(e^{i\theta},z)=V^{\alpha}_{\psi}((V^{\alpha}_{\psi})^{-1}F)
(e^{i\theta},z)=\langle (V^{\alpha}_{\psi})^{-1}F,
\,U^{\alpha}_{z}\psi_{\theta}\rangle_{A_{\alpha}^2}=
\langle F,\,V^{\alpha}_{\psi}(U^{\alpha}_{z}\psi_{\theta})
\rangle_{A_{\alpha}^2}.$$
As a result, the reproducing kernel of
$\mathcal V^{\alpha}_{\varphi}$ at $(e^{i\theta},z)$ is given by
$$V^{\alpha}_{\psi}(U^{\alpha}_{z}\psi_{\theta})
(e^{it},w)=\langle U^{\alpha}_{z}\psi_{\theta},
\, U^{\alpha}_{w}\psi_t\rangle_{A_{\alpha}^2}.$$
Let $P^{\alpha}_{\psi}$ denote the orthogonal projection
from $L^2(\T\times \D,(\alpha+1)dH)$ onto
$\mathcal V^{\alpha}_{\psi}$. Then $P^{\alpha}_{\psi}$
admits the following integral representation
$$P^{\alpha}_{\psi}F(e^{i\theta},z)=(\alpha+1)\int_0^{2\pi}
\frac{dt}{2\pi}\ind F(e^{it},w)\langle U^{\alpha}_{w}\psi_t,
\,U^{\alpha}_{z}\psi_{\theta}\rangle_{A_{\alpha}^2}d\lambda(w).$$
For $f\in L^{\infty}(\T\times\D)$, we define the Toeplitz type operator
$\mathbf T^{\psi,\alpha}_f$ on $\mathcal V^{\alpha}_{\psi}$ by
$$\mathbf T^{\psi,\alpha}_fF=P^{\alpha}_{\psi}(fF).$$

\begin{prop}\label{5.7}
Let $\psi\in A_{\alpha}^2$ with $\|\psi\|_{\alpha}=1$. For
$f\in L^{\infty}(\T\times \D)$, we have
$$(V^{\alpha}_{\psi})^{-1}\mathbf T^{\psi,\alpha}_f
V^{\alpha}_{\psi}=\mathbf L^{\psi,\alpha}_f.$$
\end{prop}

\begin{proof}
For $g,h\in A^2_{\alpha}$, we have
\begin{align*}
&\langle (V^{\alpha}_{\psi})^{-1}\mathbf T^{\psi,\alpha}_f
V^{\alpha}_{\psi}g,\,h\rangle_{A_{\alpha}^2}=\langle
\mathbf T^{\psi,\alpha}_fV^{\alpha}_{\psi}g,\,
V^{\alpha}_{\psi} h\rangle_{L^2(\T\times \D,(\alpha+1)dH)}\\
&\qquad=\langle f V^{\alpha}_{\psi}g,\,V^{\alpha}_{\psi} h
\rangle_{L^2(\T\times \D,(\alpha+1)dH)}\\
&\qquad=(\alpha+1)\int_0^{2\pi}\frac{d\theta}{2\pi}
\ind f(e^{i\theta},z)\langle g,\,U^{\alpha}_{z}\psi_{\theta}
\rangle_{A_{\alpha}^2}\langle U^{\alpha}_{z}\psi_{\theta},
\,h\rangle_{A_{\alpha}^2}\,d\lambda(z)\\
&\qquad=\langle \mathbf L^{\psi,\alpha}_fg,
\,h\rangle_{A_{\alpha}^2},
\end{align*}
which finishes the proof.
\end{proof}

We observe that the windowed Berezin transform
$B_{\alpha}^{\psi^{\alpha}}f$ of $f$ is actually the Berezin transform
of $\mathbf T_f^{\psi,\alpha}$ on $\mathcal V_{\psi}^{\alpha}$:
\begin{equation}\label{eq5.15}
B_{\alpha}^{\psi^{\alpha}}f(e^{i\theta},z)=
\langle\mathbf T^{\psi,\alpha}_f V^{\alpha}_{\psi}
(U^{\alpha}_{z}\psi_{\theta}),\,V^{\alpha}_{\psi}
(U^{\alpha}_{z}\psi_{\theta})\rangle_{L^2(\T\times \D,(\alpha+1)dH)}.
\end{equation}
We are now ready to prove Theorem~\ref{5.6}.

\begin{proof}[First proof.]
By Proposition \ref{5.7}, we have
$${\rm tr}\,(\mathbf L^{\psi^{\alpha},\alpha}_f
h(\mathbf L^{\psi^{\alpha},\alpha}_f))={\rm tr}
\,(\mathbf T^{\psi^{\alpha},\alpha}_f
h(\mathbf T^{\psi^{\alpha},\alpha}_f)).$$
It follows from \eqref{eq5.15} and \cite[Proposition 2.1]{CM}
that if $h$ is a monomial then
\begin{align*}
\int_0^{2\pi}\frac{d\theta}{2\pi}\ind
h(B_{\alpha}^{\psi^{\alpha}}f(e^{i\theta},z))f(z)
\,d\lambda(z)
&\leq {\rm tr}\,(\mathbf T^{\psi^{\alpha},\alpha}_f
h(\mathbf T^{\psi^{\alpha},\alpha}_f))\\
&\leq\int_0^{2\pi}\frac{d\theta}{2\pi}\ind h(f(z))
B_{\alpha}^{\psi^{\alpha}}f(e^{i\theta},z)\,d\lambda(z).
\end{align*}
Since $h$ is a monomial, there is a constant $C_h>0$
such that
$$|h(x)-h(y)|\leq C_h|x-y|.$$
By Theorem \ref{5.4}, we have
\begin{align*}
&\left|\int_0^{2\pi}\frac{d\theta}{2\pi}\ind
h(B_{\alpha}^{\psi^{\alpha}}f(e^{i\theta},z))f(z)
\,d\lambda(z)-\int_0^{2\pi}\frac{d\theta}{2\pi}
\ind f(z)h(f(z))\,d\lambda(z)\right|\\
&\qquad\leq\int_0^{2\pi}\frac{d\theta}{2\pi}
\ind|h(B_{\alpha}^{\psi^{\alpha}}
f(e^{i\theta},z))-h(f(z))||f(z)|\,d\lambda(z)\\
&\qquad\leq C_h\|f\|_{\infty}\int_0^{2\pi}
\frac{d\theta}{2\pi}\ind|B_{\alpha}^{\psi^{\alpha}}
f(e^{i\theta},z)-f(z)|\,d\lambda(z)\\
&\qquad=C_h\|f\|_{\infty}\|B_{\alpha}^{\psi^{\alpha}}f-f
\|_{L^1(\T\times \D)}\to 0
\end{align*}
as $\alpha\to\infty$. On the other hand, by Theorem \ref{5.4}
agian, we have
\begin{align*}
&\left|\int_0^{2\pi}\frac{d\theta}{2\pi}\ind h(f(z))
B_{\alpha}^{\psi^{\alpha}}f(e^{i\theta},z)\,d\lambda(z)-
\int_0^{2\pi}\frac{d\theta}{2\pi}\ind f(z)h(f(z))
\,d\lambda(z)\right|\\
&\qquad\leq\int_0^{2\pi}\frac{d\theta}{2\pi}\ind
|B_{\alpha}^{\psi^{\alpha}}f(e^{i\theta},z)-f(z)|
|h(f(z))|\,d\lambda(z)\\
&\qquad\leq\|h\|_{\infty}\|B_{\alpha}^{\psi^{\alpha}}f-f
\|_{L^1(\T\times \D)}\to 0
\end{align*}
as $\alpha\to\infty$. It follows that \eqref{eq5.10} holds
for all monomials.
\end{proof}

\begin{proof}[Second proof.] This proof is based on the
method used in \cite{FN1} and \cite{Wi}.

If $h(z)\equiv 1$, the desired result follows from \eqref{eq5.11}. If
$h(z)=z$, the desired result follows from \eqref{eq5.9}. In view of
Proposition \ref{5.7} and \eqref{eq5.11}, it remains for us to show
that for all $m\geq 2$ we have
\begin{equation}\label{eq5.16}
\lim_{\alpha\to\infty}\frac{{\rm tr}\,((\mathbf T^{\psi^{\alpha},
\alpha}_f)^m-\mathbf T^{\psi^{\alpha},\alpha}_{f^m})}
{\alpha+1}=0.
\end{equation}

By Lemma \ref{5.5} and Proposition \ref{5.7},
\begin{equation}\label{eq5.17}
\lim_{\alpha\to\infty}\frac{{\rm tr}\,(\mathbf T^{\psi^{\alpha},
\alpha}_f\mathbf T^{\psi^{\alpha},\alpha}_f-
\mathbf T^{\psi^{\alpha},\alpha}_{f^2})}{\alpha+1}=0.
\end{equation}
Let $P^{\alpha}_{\psi}$ denote the orthogonal projection from
$L^2(\T\times \D,(\alpha+1)dH)$ onto $\mathcal V^{\alpha}_{\psi}$
and $Q^{\alpha}_{\psi}=I-P^{\alpha}_{\psi}$. We extend
$\mathbf T^{\psi^{\alpha},\alpha}_f$ to the operator
$P^{\alpha}_{\psi}M_fP^{\alpha}_{\psi}$ on
$L^2(\T\times \D,(\alpha+1)dH)$, where $M_f$ is the multiplication
operator by $f$. Applying the same argument used in the proof of
\cite[Theorem 2.1]{FN1}, we derive \eqref{eq5.16} from \eqref{eq5.17}.
\end{proof}

Let $\lambda_i(\mathbf L_{f}^{\psi^{\alpha},\alpha})$ denote the $i$-th
singular value of $\mathbf L_{f}^{\psi^{\alpha},\alpha}$. Observe that
$$\#\left\{i:\lambda_i(\mathbf L_{f}^{\psi^{\alpha},\alpha})>\delta\right\}
={\rm tr}\,(\mathbf L^{\psi^{\alpha},\alpha}_fh_{\delta}
(\mathbf L^{\psi^{\alpha},\alpha}_f)),$$
where
$$h_{\delta}(x)=\frac{\mathbbm 1_{(\delta,\|f\|_{\infty}]}(x)}{x}.$$
Applying Theorem \ref{5.6} and a standard approximation argument
used in \cite[Corollary 2.2]{FN1}, we obtain the following corollary.
We leave the details to the interested reader.

\begin{cor}\label{5.8}
Suppose $\psi\in A^2$ with $\|\psi\|_{A^2}=1$, $\psi^{\alpha}=V^{\alpha}\psi$, and $f$ is a
non-negative function in $L^1(\D,d\lambda)\cap L^{\infty}(\D)$.
For $0<\delta\leq\|f\|_{\infty}$ we have
$$\lim_{\alpha\to \infty}\frac{\#
\left\{i:\lambda_i(\mathbf L_{f}^{\psi^{\alpha},\alpha})
>\delta\right\}}{\alpha+1}=\lambda(\{z\in\D:f(z)>\delta\}).$$
\end{cor}

We conclude the paper with the following corollary.

\begin{cor}\label{5.9}
Suppose $\psi\in A^2$ with $\|\psi\|_{A^2}=1$, $\psi^{\alpha}=V^{\alpha}\psi$, and
$f\in L^1(\D,d\lambda)\cap L^{\infty}(\D)$ is non-negative. Then
$$\lim_{\alpha\to\infty}\|\mathbf L_f^{\psi^{\alpha},\alpha}
\|_{A_{\alpha}^2}=\|f\|_{\infty}.$$
\end{cor}

\begin{proof}
It is clear that we have $\|\mathbf L_f^{\psi^{\alpha},\alpha}
\|_{A_{\alpha}^2}\leq \|f\|_{\infty}$. For any
$0<\varepsilon<\|f\|_{\infty}$, let $\delta=\|f\|_{\infty}-\varepsilon$.
By Corollary \ref{5.8}, we have
$$\lim_{\alpha\to\infty}\frac{\#\left\{i:\lambda_i
(\mathbf L_{f}^{\psi^{\alpha},\alpha})>\delta\right\}}{\alpha+1}
=\lambda(\{z\in\D:f(z)>\delta\}).$$
Then there is an $\alpha_0>0$ such that for any $\alpha>\alpha_0$
we have
$$\frac{\#\left\{i:\lambda_i(\mathbf L_{f}^{\psi^{\alpha},\alpha})
>\delta\right\}}{\alpha+1}>\lambda(\{z\in \D:f(z)>\delta\})/2.$$
Thus for $\alpha>\alpha_0$ we have
$$\|\mathbf L_f^{\psi^{\alpha},\alpha}\|_{A_{\alpha}^2}>\delta
=\|f\|_{\infty}-\varepsilon.$$
It follows that
$$\lim_{\alpha\to \infty}\|\mathbf L_f^{\psi^{\alpha},\alpha}
\|_{A_{\alpha}^2}=\|f\|_{\infty}.$$
This completes the proof.
\end{proof}

\section*{Acknowledgements}
Ma was supported by NNSF of China (Grant numbers 12171484,
12571140), NSF of Hunan Province (Grant number 2023JJ20056),
the Science and Technology Innovation Program of Hunan
Province (Grant number 2023RC3028), and Central South University
Innovation-Driven Research Programme (Grant number 2023CXQD032),
Hunan Basic Science Research Center for Mathematical Analysis (2024JC2002).
Yan was partially supported by NNSF of China (12401150) and the
Fundamental Research Funds for the Central Universities (2024CDJXY018).

\end{document}